\documentclass[11pt]{article}
\usepackage[top=0.9in,bottom=0.9in,left=1.0in,right=1.0in]{geometry}
\usepackage{amsmath,amssymb,amsthm}
\usepackage{mathtools}
\usepackage{graphicx}
\usepackage[authoryear]{natbib}
\usepackage{enumitem}
\usepackage{bm}
\usepackage{accents}
\usepackage{setspace}
\usepackage{color}
\numberwithin{equation}{section}

\newtheorem{lemma}{Lemma}
\newtheorem{theorem}{Theorem}
\newtheorem{proposition}{Proposition}
\newtheorem{corollary}{Corollary}
\theoremstyle{remark}
\newtheorem{remark}{Remark}
\numberwithin{remark}{section}

\begin{document}

\title{\textbf{Optimal ordering policy for inventory systems with
quantity-dependent setup costs}}

\author{Shuangchi He\thanks{Department of Industrial and Systems Engineering, National University of Singapore, heshuangchi@nus.edu.sg} \and
Dacheng Yao\thanks{Academy of Mathematics and Systems Science, Chinese Academy
of Sciences, dachengyao@amss.ac.cn} \and
Hanqin Zhang\thanks{Department of Decision Sciences, NUS Business School,
National University of Singapore, bizzhq@nus.edu.sg}}

\date{March 18, 2016}

\maketitle

\vspace{-0.1in}
\begin{abstract}
We consider a continuous-review inventory system in which the setup cost of each order is a general function of the order quantity and the demand process is modeled as a Brownian motion with a positive drift. Assuming the holding and shortage cost to be a convex function of the inventory level, we obtain the optimal ordering policy that minimizes the long-run average cost by a lower bound approach. To tackle some technical issues in the lower bound approach under the quantity-dependent setup cost assumption, we establish a comparison theorem that enables one to prove the global optimality of a policy by examining a tractable subset of admissible policies. Since the smooth pasting technique does not apply to our Brownian inventory model, we also propose a selection procedure for computing the optimal policy parameters when the setup cost is a step function.
\end{abstract}

\section{Introduction}
\label{sec:introduction}

Classical inventory models usually assume a setup cost when an order is placed or a production run is started to replenish the inventory. It is well known that an ordering policy of the $(s,S)$ type is optimal for the backlogging inventory problem when the setup cost is constant for any order or production quantity; see \citet{Scarf60}, \citet{Iglehart63}, and \citet{Veinott66}. Arising from various activities, order setup costs are more complex in practical inventory systems and often depend on order quantities. In this paper, we take quantity-dependent setup costs into consideration and investigate the optimal ordering policy that minimizes the long-run average cost.

Setup costs may grow as order quantities increase. For example, if an order is shipped to the buyer by multiple vehicles, a shipping fee may be charged for each of them. If a vehicle's capacity is $Q$ and the shipping fee is $F$, the total shipping cost is a nondecreasing step function of order quantity $\xi$, given by
\begin{equation}
K(\xi) = F \cdot \Big\lceil \frac{\xi}{Q} \Big\rceil.
\label{eq:nondecreasing-step}
\end{equation}
The study of stochastic inventory models with such a setup cost can be traced back to \citet{Lippman69}, where the ordering cost is assumed to be a nondecreasing, subadditive function of the order quantity. Lippman considered a periodic-review model and proved the existence of optimal ordering policies for both the finite-horizon problem and the discounted, infinite-horizon problem. It is pointed out that at the beginning of each period, it is optimal to replenish the inventory when it drops below a certain level and not to order when it is above another level. The optimal ordering decisions, however, are not specified for inventory falling in other regions. With the setup cost in \eqref{eq:nondecreasing-step}, \citet{Iwaniec79} identified a set of conditions under which a full-batch-ordering policy is optimal. \citet{AlpETAL14} allowed orders with partial batches in their policies and partially characterized the optimal ordering policy that minimizes the long-run average cost. \citet{ChaoZipkin08} considered a simple quantity-dependent setup cost
\begin{equation}
K(\xi) = F \cdot 1_{(R,\infty)}(\xi),
\label{eq:two-piece}
\end{equation}
where $1_{A}$ denotes the indicator function of $A \subset \mathbb{R}$. This formulation describes the administrative cost under a supply contract with a capacity constraint: No extra cost is incurred if the order quantity does not exceed the contract volume $R$; otherwise, the buyer is required to pay an administrative fee $F$. The authors partially characterized the optimal ordering policy for the periodic-review model and developed a heuristic policy. \citet{Caliskan-DemiragETAL12} investigated several simple forms of nondecreasing, piecewise constant setup costs, including both \eqref{eq:nondecreasing-step} and \eqref{eq:two-piece}. They also provided partial characterization for the optimal ordering policies.

As opposed to the increasing fee structure, setup costs may also decrease for large orders. To achieve economies of scale in production and distribution, suppliers in e-commerce often provide shipping discounts or free shipping for large orders. Such promotions are useful in generating additional sales. As pointed out by \citet{LewisETAL06}, the shipping policies that provide incentives for large orders may bring more profits to suppliers than standard increasing shipping fees and free shipping promotions. \citet{ZhouETAL09} analyzed a periodic-review inventory model with a free shipping option from a buyer's perspective. The setup cost in their paper is
\begin{equation}
K(\xi) = F \cdot 1_{(0,R)}(\xi),
\label{eq:free-shipping}
\end{equation}
i.e., the supplier imposes a shipping fee $F$ when the order quantity is less than $R$, but waives this charge if the order quantity exceeds $R$. They found the optimal ordering policy for the single-period problem and proposed a heuristic policy for the multiple-period model.

In practical inventory systems, order setup costs may arise from multiple activities in administration and transportation. The costs incurred by some activities may increase with the order quantity while others may decrease. As a result, the total setup cost of an order may not be monotone with respect to the order quantity. The setup cost function in this paper takes a very general form, where $ K : \mathbb{R}_{+} \rightarrow \mathbb{R} $ is assumed to satisfy the following conditions:

\begin{enumerate}[noitemsep,label=(S\arabic*),ref=S\arabic*]
\item \label{itm:S1} $ K $ is nonnegative with $K(0) = 0$;
\item \label{itm:S2} $ K $ is bounded;
\item \label{itm:S3} $ K $ has a right limit at zero, and if $ K(0+) = 0 $, $ K $ has a finite right derivative at zero;
\item \label{itm:S4} $ K $ is lower semicontinuous, i.e., 
\[  
K(\tilde{\xi}) \leq \liminf_{\xi \rightarrow \tilde{\xi}} K(\xi) \quad \mbox{for $ \tilde{\xi} > 0$.} 
\]
\end{enumerate}
Both the setup cost in \eqref{eq:two-piece} and that in \eqref{eq:free-shipping} satisfy (\ref{itm:S1})--(\ref{itm:S4}). As a technical requirement, condition (\ref{itm:S4}) ensures that the optimal average cost is attainable. The practical interpretation of this condition is as follows: If the setup cost function has a jump at $ \tilde{\xi} $, condition (\ref{itm:S4}) implies that the buyer is allowed to pay the lower fee of $ K(\tilde{\xi}-) $ and $ K(\tilde{\xi}+) $, which essentially takes account of the possibility that the buyer may adjust the order by a small quantity so as to pay a less setup fee.  Conditions (\ref{itm:S1})--(\ref{itm:S4}) are similar to the assumptions in \citet{PereraETAL15}, where the optimality of $ (s,S) $ policies is proved for economic order quantity (EOQ) models under a general cost structure.

Besides the setup cost, each order incurs a \emph{proportional cost} with rate $ k \geq 0 $. To place an order of quantity $ \xi > 0 $, the manager is required to pay an ordering cost of 
\begin{equation}
C(\xi) = K(\xi) + k\xi.
\label{eq:ordering-cost}
\end{equation}
We do \emph{not} allow multiple simultaneous orders, i.e., the ordering cost must follow \eqref{eq:ordering-cost} as long as the total order quantity at an ordering time is equal to $ \xi $. In \eqref{eq:ordering-cost}, it would be more appropriate to interpret $ K(\xi) $ as the \emph{non-proportional} part of the ordering cost, instead of the \emph{fixed} cost in the usual sense. Accordingly, $ k \xi $ represents the \emph{proportional} part, and $ k $ should be understood as the increasing rate rather than the unit price of inventory. By decomposing the ordering cost into proportional and non-proportional parts, this formulation allows for unbounded setup cost functions. For example, although the setup cost in \eqref{eq:nondecreasing-step} does not satisfy (\ref{itm:S2}), we may decompose it into
\[  
K(\xi) = \frac{F\xi}{Q} + F \Big( \Big\lceil \frac{\xi}{Q} \Big\rceil - \frac{\xi}{Q} \Big),
\]
where the first and second terms are proportional and non-proportional terms, respectively. Since the non-proportional term satisfies (\ref{itm:S1})--(\ref{itm:S4}), we may take it as the non-proportional part of the ordering cost and $ (k+F/Q)\xi $ as the proportional part. Thanks to the general form of the non-proportional cost, the ordering cost function given by \eqref{eq:ordering-cost} includes most ordering cost structures in the literature, such as ordering costs with incremental or all-unit quantity discounts; see \citet{Porteus71,Porteus72} and \citet{AltintasETAL08}. For the sake of convenience, we still refer to the non-proportional part of the ordering cost as the setup cost.

Stochastic inventory models with a general setup cost function are analytically challenging. Since the ordering cost function may be neither convex nor concave, it is difficult to identify the cost structures that can be preserved through dynamic programming. As we mentioned, the optimal ordering policy for the periodic-review model has not been fully characterized, even if the setup cost function takes the simplest form as in \eqref{eq:two-piece} or \eqref{eq:free-shipping}. The partial characterization also suggests that the optimal periodic-review policy would be complicated.

In this paper, we assume that the inventory is constantly monitored and an order can be placed at any time. To the best of our knowledge, this is the first attempt to explore optimal ordering policies for continuous-review inventory systems with quantity-dependent setup costs. In the literature on periodic-review inventory models, it is a common practice to approximate customer demand within each period by a normally distributed random variable; see, e.g., Chapter~1 in \citet{Porteus02} and Chapter~6 in \citet{Zipkin00}. Brownian motion is thus a reasonable model for demand processes in continuous-time inventory systems; see, e.g., \citet{Bather66} and \citet{Gallego90}. With a Brownian demand process, the optimal ordering policy can be obtained by solving a Brownian control problem, which turns out to be more tractable than solving a dynamic program when the setup cost is quantity-dependent. This is because with a continuous demand process, the manager is able to place an order at any inventory level as he wishes. Since future demands are independent of the history, finding the optimal ordering policy is reduced to finding constant reorder and order-up-to levels that jointly minimize the long-run average cost. In periodic-review models, by contrast, the manager is allowed to place an order only at the beginning of a period. As the inventory level varies from period to period, the optimal order decision at each period depends on both the current inventory level and the prediction of the future inventory level. The resulting dynamic program is difficult to solve when the setup cost function takes a general form; see \citet{ChaoZipkin08} and \citet{Caliskan-DemiragETAL12} for more discussion.

In our model, inventory continuously incurs a holding and shortage cost that is a convex function of the inventory level. With the aforementioned assumptions, we prove that an $(s,S)$ policy can minimize the long-run average cost, and that $ (s^{\star}, S^{\star}) $, the optimal reorder and order-up-to levels, can be obtained by solving a nonlinear optimization problem. When the setup cost function satisfies $ K(0+) = 0 $, we prove that $s^{\star} = S^{\star} < 0$ holds under certain conditions, in which case the optimal ordering policy becomes a base stock policy that maintains inventory above a fixed shortage level.

Brownian inventory models were first introduced by \citet{Bather66}. In his pioneering work, Bather studied the impulse control of Brownian motion that allows upward adjustments. Assuming a constant setup cost and a convex holding and shortage cost, he obtained the $(s,S)$ policy that minimizes the long-run average cost. Bather's results have been extended to more general settings by a number of studies under the constant setup cost assumption. Among them, the $(s,S)$ policy that minimizes the discounted cost was obtained by \citet{Sulem86} with a piecewise linear holding and shortage cost, and by \citet{Benkherouf07} with a convex holding and shortage cost. \citet{Bar-IlanSulem95} obtained the optimal $(s,S)$ policy for a Brownian inventory model that allows for constant lead times, and \citet{MuthuramanETAL15} extended their results to a Brownian model with stochastic lead times. \citet{BensoussanETAL05} and \citet{BenkheroufBensoussan09} studied a stochastic inventory model where the demand is a mixture of a Brownian motion and a compound Poisson process; the optimal policy for this model is of the $(s,S)$ type again. Using the fluctuation theory of L\'{e}vy processes, \citet{Yamazaki2013} generalized their results to spectrally positive L\'{e}vy demand processes. In the above papers, the optimal ordering policies are obtained by solving a set of quasi-variational inequalities (QVIs) deduced from the Bellman equation. For computing the optimal parameters, one needs to impose additional smoothness conditions at the reorder and order-up-to levels. This technique, widely known as \emph{smooth pasting}, is essential to solve a Brownian control problem by the QVI approach. See \citet{Dixit93} for a comprehensive account of smooth pasting and its applications. 

\citet{HarrisonETAL83} studied the impulse control of Brownian motion allowing both upward and downward adjustments, for which a control band policy is proved optimal under the discounted cost criterion. In that paper, the authors adopted a two-step procedure which has become a widely used approach to solving Brownian control problems: In the first step, one establishes a lower bound for the cost incurred by an arbitrary admissible policy; such a result is often referred to as a \emph{verification theorem}. In the second step, one searches for an admissible policy to achieve this lower bound; the obtained policy, if any, must be optimal. The technique of smooth pasting is also a standard component of the lower bound approach. By imposing additional smoothness conditions at the boundary of a control policy, one may obtain the optimal control parameters through solving a \emph{free boundary problem}. Following this approach, \citet{OrmeciETAL08} obtained the optimal control band policy under the long-run average cost criterion. Both \citet{HarrisonETAL83} and \citet{OrmeciETAL08} assumed a constant setup cost and a piecewise linear holding and shortage cost. Their results were extended by \citet{DaiYao13a,DaiYao13b}, who allowed for a convex holding and shortage cost and obtained the optimal control band policies under both the average and discounted cost criteria. Using the lower bound approach, \citet{HarrisonTaksar83} and \citet{Taksar85} studied the two-sided instantaneous control of Brownian motion, where double barrier policies are proved optimal under different cost criteria. \citet{BaurdouxYamazaki15} extended the optimality of double barrier policies to spectrally positive L\'{e}vy demand processes. The lower bound approach was also adopted by \citet{WuChao14}, who studied optimal production policies for a Brownian inventory system with finite production capacity, and by \citet{YaoETAL15}, who studied optimal ordering policies with a concave ordering cost. The optimal policies in these two papers are of the $(s,S)$ type.

We follow the lower bound approach in this paper, while new issues arise from our Brownian model. We establish a verification theorem in Proposition~\ref{prop:lower-bound}. It states that if there exists a continuously differentiable function $f$ and a positive number $\underaccent{\bar}{\nu}$ such that they jointly satisfy some differential inequalities, the long-run average cost under any admissible policy must be at least~$\underaccent{\bar}{\nu}$. We derive this lower bound using It\^{o}'s formula, as in the previous studies by \citet{HarrisonETAL83}, \citet{OrmeciETAL08}, and \citet{DaiYao13a,DaiYao13b}. The Brownian model in those papers allows both upward and downward adjustments, so a control band policy is expected to be optimal. Under such a policy, the inventory level is confined within a finite interval and the associated relative value function is Lipschitz continuous. This fact allows them to assume $f$ to be Lipschitz continuous in the verification theorems. With this assumption, one can prove the lower bound by relying solely on It\^{o}'s formula. In our model, however, only upward adjustments are allowed. The optimal policy is expected to be an $(s,S)$ policy whose relative value function is not Lipschitz continuous. Without the Lipschitz assumption, it is difficult to prove the verification theorem in a direct manner. This problem was also encountered by \citet{WuChao14} and \citet{YaoETAL15}. In their papers, the lower bound results are established for a subset, rather than all of admissible policies; accordingly, the proposed $(s,S)$ policies are proved optimal within the same subset of policies.

We prove a comparison theorem to tackle this issue. Theorem~\ref{thm:dense} in this paper states that for any admissible policy, we can always find an admissible policy that has a finite order-up-to bound and whose long-run average cost is either less than or arbitrarily close to the average cost incurred by the given policy. In other words, if an ordering policy can be proved optimal within the set of policies having order-up-to bounds, it must be optimal among all admissible policies. This result allows us to prove the verification theorem by examining an arbitrary admissible policy that is subject to a finite order-up-to bound. With an order-up-to bound, we no longer require $f$ to be Lipschitz continuous for establishing the verification theorem by It\^{o}'s formula.

For an $(s,S)$ policy, the associated relative value function and the resulting long-run average cost jointly satisfy a second-order ordinary differential equation along with some boundary conditions; see Proposition~\ref{prop:sS} for the solution to this equation. We use this relationship to compute the optimal reorder and order-up-to levels. In the literature, the optimal $(s,S)$ policies for Brownian models with a constant setup cost are obtained by imposing smooth pasting conditions on the ordinary differential equations; see \citet{Bather66}, \citet{Sulem86}, \citet{Bar-IlanSulem95}, and \citet{WuChao14}. Unfortunately, our Brownian model does not preserve this property because the general setup cost function has imposed a quantity constraint on each setup cost value. With these constraints, the smooth pasting conditions may no longer hold at the optimal reorder and order-up-to levels. Without definite boundary conditions, we can neither define a free boundary problem nor solve the QVI problem for the optimal $ (s,S) $ policy. To obtain the optimal ordering policy, we need to minimize the long-run average cost by solving a nonlinear optimization problem. When the setup cost is a step function, we develop a policy selection algorithm for computing the optimal policy parameters.

The contribution of this paper is twofold. First, by assuming a Brownian demand process, we obtain the optimal ordering policy for inventory systems with quantity-dependent setup costs, filling a long-standing research gap. The optimality of $(s,S)$ policies is extended to a significantly more general cost structure. Although the optimal policy is obtained using a continuous-review model, it will shed light on periodic-review models, presumably serving as a simple and near-optimal solution. Second, the comparison theorem and the policy selection algorithm complement the well-established lower bound approach to solving Brownian control problems. Theorem~\ref{thm:dense} in this paper enables one to prove the optimality of a policy by examining a tractable subset, instead of all admissible policies. The constructive proof of this theorem can be extended to similar comparison results with minor modification. Using modified comparison theorems, we expect that both the production policy proposed by \citet{WuChao14} and the ordering policy proposed by \citet{YaoETAL15} will be proved globally optimal. Besides inventory control, our approach may also be used for solving Brownian control problems arising from financial management (see, e.g., \citealt{Constantinides76} and \citealt{Paulsen08}), production systems (\citealt{Wein92}, \citealt{VeatchWein96}, and \citealt{AtaETAL05}), and queueing control (\citealt{Ata06} and \citealt{RubinoAta09}).

The rest of this paper is organized as follows. We introduce the Brownian inventory model in \S\ref{sec:model} and present the main results in \S\ref{sec:main}. A lower bound is derived in \S\ref{sec:lower-bound} for the long-run average cost under an arbitrary admissible policy. The relative value function and the average cost under an $(s,S)$ policy are analyzed in \S\ref{sec:sS}. Using these results, we prove the optimality of the proposed policy in \S\ref{sec:proof-1}. Section~\ref{sec:proof-3} is dedicated to the proof of Theorem~\ref{thm:dense}, which enables us to investigate an optimal policy within a subset of admissible policies. We introduce a policy selection algorithm in \S\ref{sec:algorithm}, for obtaining the optimal ordering policy when the setup cost is a step function. The paper is concluded in \S\ref{sec:conclusion}, and we leave the proofs of technical lemmas to the appendix.

Let us close this section with frequently used notation. Let $\varphi$ be a real-valued function defined on $\mathbb{R}$. We use $ \Delta \varphi(t) $ to denote its increment at $ t $, i.e., $\Delta \varphi(t)=\varphi(t+) - \varphi(t-)$, if the one-sided limits exist. We use $\varphi'(t)$ and $\varphi''(t)$ to denote its first and second derivatives at $t$, respectively.

\section{Brownian inventory model}
\label{sec:model}

Consider a continuous-time inventory system whose inventory level at time $t\geq 0$ is denoted by $Z(t)$. We allow $Z(t)$ to be less than zero, in which case $|Z(t)|$ is interpreted as the back order or shortage level. We assume that all unsatisfied demands will be back-ordered and that the lead time of each order is zero. Let $D(t)$ and $Y(t)$ be the cumulative demand quantity and the cumulative order quantity during $[0,t]$, respectively. The inventory level at time $t\geq 0$ is given by
\[
Z(t) = x - D(t) + Y(t),
\]
where $x$ is a real number. We refer to $ Z = \{ Z(t):t \geq 0 \} $ as the \emph{inventory process}, and put $Z(0-)=x$ which is interpreted as the initial inventory level. We assume that the cumulative demand process $D=\{D(t):t\geq 0\}$ is a Brownian motion that starts from $D(0)=0$ and has drift $\mu>0$ and variance $\sigma^{2}>0$. In other words, $D$ has the representation
\[
D(t) = \mu t-\sigma B(t),
\]
where $B=\{B(t):t\geq 0\}$ is a standard Brownian motion defined on a filtered probability space $(\Omega, \mathcal{F},\mathbb{F},\mathbb{P})$ with filtration $\mathbb{F}=\{\mathcal{F}(t):t\geq 0\}$. Then, the inventory level at time $t$ can be written as
\begin{equation}
Z(t) = X(t)+ Y(t),
\label{eq:Z}
\end{equation}
where
\begin{equation}
X(t) = x-\mu t +\sigma B(t).
\label{eq:X}
\end{equation}
We refer to $ X = \{ X(t) : t \geq 0 \} $ as the \emph{uncontrolled} inventory process. The system manager replenishes the inventory according to a non-anticipating ordering policy, which is specified by the cumulative order process $Y=\{Y(t):t\geq 0\}$. More specifically, an ordering policy is said to be \emph{admissible} if $Y$ satisfies the following three conditions. First, for each sample path $\omega\in\Omega$, $Y(\omega, \cdot)$ is a nondecreasing function that is right-continuous on $[0,\infty)$ and has left limits on $(0,\infty)$. Second, $Y(t)\geq 0$ for all $t\geq 0$. Third, $Y$ is adapted to $\mathbb{F}$, i.e., $Y(t)$ is $\mathcal{F}(t)$-measurable for all $t\geq 0$. We use $\mathcal{U}$ to denote the set of all admissible ordering policies, or equivalently, the set of all cumulative order processes that satisfy the above three conditions. With the convention $Y(0-)=0$, an admissible policy $Y$ is said to \emph{increase} at time $t\geq 0$ if $Y(u)-Y(t-)>0$ for all $u>t$. We call $t$ an \emph{ordering time} if $Y$ increases at $t$. Let $\mathcal{I}(t)$ be the cardinality of the set $\{u\in[0,t]: Y\mbox{ increases at } u\}$, which is interpreted as the number of orders placed by time $t$. Moreover, $t$ is said to be a \emph{jump time} if $\Delta Y(t)>0$. Let
\begin{equation}
Y^{c}(t) = Y(t)-\sum_{0\leq u \leq t}\Delta Y(u).
\label{eq:continuous-part}
\end{equation}
Then, $Y^{c}=\{Y^{c}(t):t\geq 0\}$ is the \emph{continuous part} of $Y$.

Each order incurs an ordering cost given by \eqref{eq:ordering-cost}, with $ k \geq 0 $ and $ K $ satisfying (\ref{itm:S1})--(\ref{itm:S4}). If $ K(0+) > 0 $, we only need to consider the policies that place finitely many orders over a finite time interval, i.e., $ \mathcal{I}(t) < \infty $ almost surely for $ t > 0 $, because otherwise, either the cumulative ordering cost or the cumulative holding and shortage cost will be infinite by time $ t $. In other words, when $ K(0+) > 0 $, we consider $ Y $ that is a piecewise constant function on almost all sample paths, which implies that $ Y^{c}(t) = 0 $ for all $ t > 0 $ almost surely. Therefore, the cumulative ordering cost during $ [0,t] $ is given by
\begin{equation}
\mathcal{C}_{Y}(t) = \sum_{0 \leq u \leq t} K(\Delta Y(u)) + k Y(t).
\label{eq:ordering-K-positive}
\end{equation}
When $ K(0+) = 0 $, the manager may also exert inventory control through the continuous part of $ Y $, which may no longer be a zero process. To analyze the setup cost incurred by $ Y^{c} $, let us put
\begin{equation}
\ell = \liminf_{\xi \downarrow 0} \frac{K(\xi)}{\xi}.
\label{eq:ell}
\end{equation}
By (\ref{itm:S3}), $ \ell $ is the right derivative of $ K $ at zero if $ K(0+) = 0$. We would thus interpret $ \ell $ as the unit setup cost when the order quantity is infinitesimal. Besides a proportional cost of $ k $, every unit increment of $ Y^{c} $ incurs a setup cost of $ \ell $. Hence, the cumulative ordering cost during $ [0,t] $ is
\begin{equation}
\mathcal{C}_{Y}(t) = \sum_{0 \leq u \leq t} K(\Delta Y(u)) + k Y(t) + \ell Y^{c} (t).
\label{eq:ordering-K}
\end{equation}
Note that $ \ell = \infty $ if $ K(0+) > 0 $. Following the convention that $ 0 \cdot \infty = 0 $, we may take \eqref{eq:ordering-K-positive} as a special case of the cumulative ordering cost given by \eqref{eq:ordering-K}.

In addition to the ordering cost, the system incurs a holding and shortage cost that is charged at rate $h(z)$ when the inventory level is $z$. More specifically, $h(z)$ is the inventory holding cost per unit of time for $z\geq 0$, and is the shortage cost of back orders per unit of time for $z<0$. The cumulative holding and shortage cost by time $t$ is thus given by 
\begin{equation}
\mathcal{H}_{Y}(t) = \int_{0}^{t} h(Z(u))\,\mathrm{d}u,
\label{eq:holding}
\end{equation}
which depends on the ordering policy through the inventory process. We assume that $h$ satisfies the following conditions:
\begin{enumerate}[noitemsep,label=(H\arabic*),ref=H\arabic*]
\item \label{itm:H1} $h(0)=0$;
\item \label{itm:H2} $h$ is a convex function;
\item \label{itm:H3} $h$ is continuously differentiable except at $z=0$;
\item \label{itm:H4} $h'(z)>0$ for $z>0$ and $h'(z)<0$ for $z<0$;
\item \label{itm:H5} $h$ is polynomially bounded, i.e., there exists a positive integer $a$ and two positive numbers $b_{0}$ and $b_{1}$ such that $h(z) \leq b_{0}+b_{1}|z|^{a}$ for all $z\in\mathbb{R}$.
\end{enumerate}
In particular, with $\beta_{1},\beta_{2}, \beta>0$, both the piecewise linear cost 
\[
h(z) =
\begin{cases}
\beta_{1}z & \mbox{for }z\geq 0\\
-\beta_{2}z & \mbox{for }z<0
\end{cases}
\]
and the quadratic cost $h(z)=\beta z^{2}$ satisfy (\ref{itm:H1})--(\ref{itm:H5}).

Given the initial inventory level $x$ and the ordering policy $ Y \in \mathcal{U} $, the long-run average cost is defined as
\[  
\operatorname{AC}(x,Y) = \limsup_{t\rightarrow \infty}\frac{1}{t}\mathbb{E}_{x}
[ \mathcal{C}_{Y}(t) + \mathcal{H}_{Y}(t)],
\]
where $\mathbb{E}_{x}[\cdot]$ is the expectation conditioning on the initial inventory level $Z(0-)=x$. By \eqref{eq:ordering-K}--\eqref{eq:holding}, the long-run average cost is given by
\begin{equation}
\operatorname{AC}(x,Y) = \limsup_{t\rightarrow \infty}\frac{1}{t}\mathbb{E}_{x}
\Big[\int_{0}^{t}h(Z(u))\,\mathrm{d}u+\sum_{0\leq u\leq t}K(\Delta Y(u)
) + kY(t) + \ell Y^{c}(t)\Big].
\label{eq:AC}
\end{equation}

When $K(0+)>0$, we only need to consider ordering policies having piecewise constant sample paths. Such a policy can be specified by a sequence of pairs $\{(\tau_{j},\xi_{j}):j=0,1,\ldots\}$ where $\tau_{j}$ is the $j$th ordering time and $\xi_{j}$ is the quantity of the $j$th order. By convention, we set $\tau_{0}=0$ and let $\xi_{0}$ be the quantity of the order placed at time zero ($\xi_{0}=0$ if no order is placed). With this sequence, the ordering policy $Y$ can be specified by $ Y(t) = \sum_{j=0}^{\mathcal{J}(t)} \xi_{j} $, where $\mathcal{J}(t) = \max\{j\geq 0: \tau_{j}\leq t \}$. On the other hand, if the ordering policy $Y$ is given, we can obtain each ordering time by
\[
\tau_{j} = \inf\{t>\tau_{j-1}:Y(t)>Y(t-)\}\quad\mbox{for } j = 1,2,\ldots
\]
and each order quantity by
\[
\xi_{j} = Y(\tau_{j})-Y(\tau_{j}-)\quad\mbox{for } j=0,1,\ldots.
\]
Therefore, finding an optimal ordering policy when $ K(0+) > 0 $ is equivalent to specifying a sequence of optimal ordering times and order quantities $\{(\tau_j,\xi_j): j=0,1,\ldots\}$, which turns out to be an \emph{impulse control} problem for the Brownian model. For the ordering policy $Y$ to be adapted to $\mathbb{F}$, we require each $\tau_{j}$ to be an $\mathbb{F}$-stopping time and $\xi_{j}$ to be $\mathcal{F} (\tau_{j})$-measurable.

When the setup cost has $ K(0+) = 0 $, the manager may adjust the inventory level using the continuous part of $ Y $ without incurring infinite costs. If $ Y^{c} $ is not a zero process, we will have $ \mathcal{I}(t) = \infty $ for some $ t > 0 $ with a positive probability. It may happen that the optimal ordering policy has continuous sample paths except for a possible jump at time zero. In this case, the ordering problem becomes an \emph{instantaneous control} problem for the Brownian model.

\section{Main results}
\label{sec:main}

The main results of this paper are presented in this section. Theorem~\ref{thm:main-results} states that with a setup cost that satisfies (\ref{itm:S1})--(\ref{itm:S4}) and a holding and shortage cost that satisfies (\ref{itm:H1})--(\ref{itm:H5}), the optimal ordering policy for the Brownian inventory model is an $(s,S)$ policy with $s\leq S$. In addition, the optimal reorder and order-up-to levels $ (s^{\star}, S^{\star}) $ satisfy $ s^{\star} < S^{\star} $ if $ K(0+) > 0 $, and satisfy $ s^{\star} \leq S^{\star} $ if $ K(0+) = 0$. As a special case, the optimal ordering policy becomes a base stock policy when $ s^{\star} = S^{\star} $. We also provide a comparison result in Theorem~\ref{thm:dense}, which is a technical tool for proving the first theorem by the lower bound approach.

Under an $(s,S)$ policy, as long as the inventory level drops below $s$, the manager places an order that replenishes the inventory to level $S$ immediately. We use $U(s,S)$ to denote this policy. Clearly, $U(s,S)\in \mathcal{U}$ for $s\leq S$. An $(s,S)$ policy with $s<S$ can be specified by the sequence of pairs $\{(\tau_{j},\xi_{j}): j=0,1,\ldots\}$ as follows. With $ \tau_{0}=0 $ and 
\[  
\xi_{0}=
\begin{cases}
S-x & \mbox{if $ x\leq s $,}\\
0 & \mbox{if $ x>s $,}
\end{cases}
\]
the $j$th order is placed at time $ \tau_{j} = \inf\{t>\tau_{j-1}:Z(t-)\leq s\} $ with a constant quantity $ \xi_{j} = S-s $.

If the reorder and order-up-to levels are equal, the $ (s,S) $ policy becomes a \emph{base stock policy}. Under the base stock policy, if the initial inventory level $x$ is below the base stock level $s$, the manager places an order of quantity $s-x$ at time zero that replenishes the inventory to $Z(0)=s$; otherwise, the manager does not order at time zero. After that, whenever the inventory level drops below $s$, the manager brings it back to $s$ immediately. Such a policy is well defined for our Brownian model, and the inventory process under that has an analytic expression---see the lemma below.
\begin{lemma}
\label{lem:Skorokhod-Map}
Let $s$ be a real number and $\mathbb{C}[0,\infty)$ the set of real-valued continuous functions on $[0,\infty)$. Then for each $\phi\in\mathbb{C}[0,\infty)$, there exists a unique pair of functions $(\eta,\zeta)\in\mathbb{C}[0,\infty)\times\mathbb{C}[0,\infty)$ such that \textnormal{(i)} $\eta$ is nondecreasing with $\eta(0) = (s - \phi(0))^{+};$ \textnormal{(ii)} $\zeta(t) = \phi(t)+\eta(t)\geq s$ for $t\geq 0;$ \textnormal{(iii)} $\eta$ increases only when $\zeta(t)=s$, i.e.,
\[
\int_{0}^{\infty}(\zeta(t)-s)\,\mathrm{d}\eta(t)=0.
\]
Specifically,
\[
\eta(t) = \sup_{0\leq u\leq t} (s-\phi(u))^{+}\quad\mbox{for }t\geq 0.
\]
\end{lemma}
This lemma is a modified version of Proposition~2.1 in \citet{Harrison13}. The proof is similar and thus omitted. Under the base stock policy, the inventory process in our Brownian model becomes a reflected Brownian motion with lower reflecting barrier at $s$. By Lemma~\ref{lem:Skorokhod-Map}, the cumulative order quantity during $[0,t]$ is
\[
Y(t) = \sup_{0\leq u\leq t} (s-X(u))^{+},
\]
where $X(u)$ is given by (\ref{eq:X}). Clearly, $Y$ is admissible for each $s\in\mathbb{R}$. Because $Y$ has continuous sample paths, for each $t>0$, there are infinite ordering times in $[0,t]$ with a positive probability.

Before stating the main theorem, let us introduce a proposition that characterizes the optimal policy of the $ (s,S) $ type. In particular, the long-run average cost under an $ (s,S) $ policy has an analytic expression, which is given by \eqref{eq:gamma} below and will be proved in \S\ref{sec:sS} (see Proposition~\ref{prop:sS}).

\begin{proposition}
\label{prop:z-s-S-star}	
Assume that the setup cost $ K $ satisfies \textnormal{(\ref{itm:S1})--(\ref{itm:S4})} and that the holding and shortage cost $ h $	satisfies \textnormal{(\ref{itm:H1})--(\ref{itm:H5})}. Let
\begin{equation}
\gamma(s,S)=
\begin{dcases}
k\mu + \frac{K(S-s)\mu}{S-s} + \frac{\lambda}{S-s} \int_{s}^{S} \int_{0}^{\infty}h(u+y)\mathrm{e}^{-\lambda u}\,\mathrm{d}u \,\mathrm{d}y & \mbox{if $s<S$},\\
(k + \ell)\mu + \lambda \int_{0}^{\infty}h(u+s)\mathrm{e}^{-\lambda u}\,\mathrm{d}u &\mbox{if $s=S$,}
\end{dcases}
\label{eq:gamma}
\end{equation}
where $ \lambda = 2\mu / \sigma^{2} $. Then, there exists $ (s^{\star}, S^{\star}) \in \mathbb{R}^{2} $ such that
\begin{equation}
\gamma(s^{\star}, S^{\star}) = \inf\{ \gamma(s,S) : s \leq S \}.
\label{eq:s*S*}
\end{equation}
If $ K(0+) > 0$, the minimizer $ (s^{\star}, S^{\star}) $ satisfies $ s^{\star} < z^{\star} < S^{\star} $, where $ z^{\star} $ is the unique solution to
\begin{equation}
\lambda \int_{0}^{\infty} h(u+z^{\star}) \mathrm{e}^{-\lambda u} \, \mathrm{d}u = h(z^{\star})
\label{eq:optimal-barrier}
\end{equation}
and satisfies $ z^{\star} < 0 $; if $ K(0+) = 0 $, the minimizer satisfies either $ s^{\star} < z^{\star} < S^{\star} $ or $ s^{\star} = z^{\star} = S^{\star} $.
\end{proposition}
\begin{remark}
\label{rmk:sS}
The long-run average cost under the $ (s,S) $ policy is given by $ \gamma(s,S) $ in \eqref{eq:gamma}. When $ s < S $, the first expression in \eqref{eq:gamma} can be interpreted as follows. Since the quantity of each order is $ S-s $, the long-run average proportional and setup costs are $ k\mu $ and $ K(S-s)\mu/(S-s) $, respectively. The inventory process under the $ (s,S) $ policy is regenerative. Within each cycle, the trajectory of $ Z $ is identical to that of a Brownian motion starting from $ S $ with drift $ -\mu $ and variance $ \sigma^{2} $, so a cycle length has the same distribution as the first hitting time of $ s $ by $ X $ in \eqref{eq:X} with $ X(0) = S $. More specifically, assuming $ X(0) = x$, let us put
\begin{equation}
T(y)=\inf \{t\geq 0: X(t) = y \}
\quad \mbox{and} \quad
H_{x}(y) = \mathbb{E}_{x} \Big[\int_{0}^{T(y)}h(X(u)) \,\mathrm{d}u\Big].
\label{eq:tau}
\end{equation}
Then with $ X(0) = S $, the length of a cycle can be represented by $T(s)$ and the expected holding and shortage cost during a cycle is $H_{S}(s)$. The long-run average holding and shortage cost is thus equal to $H_{S}(s)/\mathbb{E}_{S}[T(s)]$ (see, e.g., Theorem~VI.3.1 in \citealt{Asmussen03}). By Theorem~5.32 in \citet{Serfozo09}, $\mathbb{E}_{S}[T(s)] = (S-s)/\mu $. The formula of $H_{x}(y)$ can be found in \S15.3 in \citet{KarlinTaylor81}, where
\[
H_{S}(s)
=\frac{\lambda}{\mu} \int_{s}^{S}\int_{0}^{\infty}h(u+y)
\mathrm{e}^{-\lambda u}\,\mathrm{d}u\,\mathrm{d}y.
\]
Hence, the third term on the right side is the long-run average holding and shortage cost. When $ K(0+) = 0 $ in (\ref{itm:S3}), by taking $ S-s \rightarrow 0 $, the long-run average cost of the $ (s,S) $ policy converges to
\[  
\gamma(s,s) = (k + \ell)\mu + \lambda \int_{0}^{\infty}h(u+s)\mathrm{e}^{-\lambda u}\,\mathrm{d}u.
\]
Since the $ (s,S) $ policy turns out to be a base stock policy when $ s = S $, the second expression in \eqref{eq:gamma} is the long-run average cost under the base stock policy with base stock level $ s $. 
\end{remark}

\begin{remark}
\label{rmk:base-stock}
	
The pair $ (s^{\star}, S^{\star}) $ that satisfies \eqref{eq:s*S*} specifies the reorder and order-up-to levels (which may not be unique) for the optimal $ (s,S) $ policy. When $ K(0+) = 0 $, the optimal $ (s,S) $ policy may be a base stock policy whose base stock level $ z^{\star} $ is specified by \eqref{eq:optimal-barrier}. Since $z^{\star}<0$, the inventory under the optimal base stock policy is maintained above a fixed shortage level. Regulated by the (slightly) negative base stock level, the inventory will fluctuate in a neighborhood of zero. By \eqref{eq:gamma} and \eqref{eq:optimal-barrier}, the minimum long-run average cost is equal to
\[
\gamma(z^{\star}, z^{\star}) = (k+\ell)\mu + h(z^{\star}).
\]
The optimal base stock level can be interpreted as follows. As discussed in \S\ref{sec:model}, the long-run average ordering cost must be $(k+\ell)\mu$ under any base stock policy. The optimal base stock policy should thus minimize the average holding and shortage cost. As a reflected Brownian motion with a negative drift, $Z$ will reach a steady state as time goes by. Let $Z(\infty)$ be the steady-state inventory level. If the base stock level is $s$, $Z(\infty)-s$ follows an exponential distribution with rate $\lambda=2\mu/\sigma^2$ (see, e.g., Proposition~6.6 in \citealt{Harrison13}). The resulting long-run average holding and shortage cost is given by
\[
H(s) = \int_{0}^{\infty} h(u+s)\cdot \lambda \mathrm{e}^{-\lambda u}\,
\mathrm{d}u.
\]
By setting the first derivative of $H$ equal to zero, the optimal base stock 	level can be obtained by solving \eqref{eq:optimal-barrier}, from which we also have $H(z^{\star})=h(z^{\star})$. Therefore, $ (k+\ell)\mu + h(z^{\star}) $ is the long-run average cost under the optimal base stock policy.
\end{remark}

\begin{remark}
Although the optimal base stock policy incurs less holding and shortage cost than any $ (s,S) $ policy with $ s < S $, there may exist some $ s < S $ such that the $ (s,S) $ policy with these parameters incurs less setup cost, i.e., $ K(S-s)/(S-s) < \ell $. When $ K(0+) = 0$, the optimal reorder and order-up-to levels may either satisfy $ s^{\star} < z^{\star} < S^{\star} $ or satisfy $ s^{\star} = z^{\star} = S^{\star} $.
\end{remark}

Let 
\[
\nu^{\star} = \inf\{ \operatorname{AC}(x,Y) : x \in \mathbb{R},\ Y \in \mathcal{U} \},
\]
where $ \operatorname{AC}(x,Y) $ is the long-run average cost given by \eqref{eq:AC}. Theorem~\ref{thm:main-results} states the optimality of $ (s,S) $ policies among all admissible policies. Under the average cost criterion, neither the optimal ordering policy nor the minimum long-run average cost depend on the initial inventory level.

\begin{theorem}
\label{thm:main-results}
Assume that the setup cost $ K $ satisfies \textnormal{(\ref{itm:S1})--(\ref{itm:S4})} and that the holding and shortage cost $ h $	satisfies \textnormal{(\ref{itm:H1})--(\ref{itm:H5})}. Then, with $ (s^{\star},S^{\star}) $ determined by \eqref{eq:s*S*}, $ U(s^{\star}, S^{\star}) $ is an optimal ordering policy that minimizes the long-run average cost, i.e., $ \nu^{\star}  = \gamma(s^{\star}, S^{\star}) $ with $ \gamma $ given by \eqref{eq:gamma}.
\end{theorem}

The second theorem is a critical result for proving Theorem~\ref{thm:main-results} by the lower bound approach, playing an important role in establishing the verification theorem (see Proposition~\ref{prop:lower-bound} in \S\ref{sec:lower-bound}). It implies that an ordering policy that is optimal within the set of policies having order-up-to bounds must be optimal in all admissible policies. Since policies subject to order-up-to bounds are analytically tractable, it is more convenient to prove the optimality within these policies.

For $m=1,2,\ldots$, let
\[
\mathcal{U}_{m} = \{Y\in\mathcal{U}:Z(t)\leq m\mbox{ for all ordering time }t\},
\]
which is the set of admissible policies with an order-up-to bound at $m$. Clearly, $U(s,S)\in\mathcal{U}_{m}$ if $s\leq S\leq m$. Because of the Brownian demand process, it is possible that $Z(t) > m$ under a policy in $\mathcal{U}_{m}$ if $t$ is not an ordering time.

\begin{theorem}
\label{thm:dense}
Assume that the setup cost $ K $ satisfies \textnormal{(\ref{itm:S1})--(\ref{itm:S3})} and that the holding and shortage cost $h$ is nondecreasing on $[0,\infty)$. Then, for any admissible policy $Y$, there exists a sequence of admissible policies $\{Y_{m}\in \mathcal{U}_{m}: m=1,2,\ldots \}$ such that
\begin{equation}
\lim_{m\rightarrow\infty} \operatorname{AC}(x,Y_{m})\leq \operatorname{AC}(x,Y).
\label{eq:dense}
\end{equation}
\end{theorem}
Let $\bar{\mathcal{U}} = \bigcup_{m=1}^{\infty} \mathcal{U}_{m}$ be the set of admissible policies subject to order-up-to bounds. Theorem~\ref{thm:dense} implies that a policy that is optimal in $\bar{\mathcal{U}}$ must be optimal in $\mathcal{U}$. We will prove this theorem in \S\ref{sec:proof-3}.

\section{A lower bound for long-run average costs}
\label{sec:lower-bound}

In this section, we establish a lower bound for the long-run average cost under an arbitrary admissible policy. This lower bound is specified by differential inequalities with respect to a relative value function. In the lower bound approach, such a result is referred to as a verification theorem.

\begin{proposition}
\label{prop:lower-bound}
Assume that $ K $ satisfies \textnormal{(\ref{itm:S1})--(\ref{itm:S3})} and that $h$ satisfies \textnormal{(\ref{itm:H1})--(\ref{itm:H5})}. Let $f:\mathbb{R} \rightarrow \mathbb{R}$ be a continuously differentiable function with $f'$ absolutely continuous. Assume that
\begin{equation}
f(z_{1})-f(z_{2}) \geq -K(z_{1}-z_{2})-k(z_{1}-z_{2})\quad\mbox{for }z_{1}>z_{2},
\label{eq:f-d}
\end{equation}
and that there exists a positive integer $ d $ and two positive numbers $a_{0}$ and $a_{1}$ such that
\begin{equation}
|f'(z)|<a_{0} \quad\mbox{for }z < 0
\label{eq:fprime-n}
\end{equation}
and
\begin{equation}
|f'(z)|<a_{0}+a_{1}z^{d}\quad\mbox{for }z\geq 0.
\label{eq:fprime-p}
\end{equation}
Let $\Gamma$ be the generator of $X$ in \eqref{eq:X}, i.e.,
\[
\Gamma f(z) = \frac{1}{2}{\sigma^{2}f''(z)}-\mu f'(z).
\]
Assume that there exists a positive number $\underaccent{\bar}{\nu}$ that satisfies
\begin{equation}
\Gamma f(z)+h(z)\geq \underaccent{\bar}{\nu} \quad\mbox{for }z\in\mathbb{R}\mbox{ such that }
f''(z) \mbox{ exists.}
\label{eq:lower-bound}
\end{equation}
Then, $ \operatorname{AC}(x,Y)\geq \underaccent{\bar}{\nu} $ for $ x\in\mathbb{R} $ and $ Y\in\mathcal{U} $, where $\operatorname{AC}(x,Y)$ is given by \eqref{eq:AC}.
\end{proposition}

If we can find an ordering policy whose relative value function satisfies all the assumptions on~$f$ and whose long-run average cost  $\underaccent{\bar}{\nu}$ satisfies \eqref{eq:lower-bound}, then by Proposition~\ref{prop:lower-bound}, this policy must be optimal in all admissible policies. To prove Proposition~\ref{prop:lower-bound}, we need two technical lemmas about inventory processes subject to an order-up-to bound. 

For a given positive integer $m$, let
\begin{equation}
Y^{m}(t) = \sup_{0\leq u\leq t} (m-X(u))^{+}
\quad\mbox{and}\quad
Z^{m}(t) = X(t) + Y^{m}(t).
\label{eq:Z-m}
\end{equation}
By Lemma~\ref{lem:Skorokhod-Map}, $Y^{m} = \{ Y^{m}(t) : t \geq 0 \}$ is the base stock policy with base stock level $m$, under which $Z^{m} = \{ Z^{m}(t) : t \geq 0 \}$ is a reflected Brownian motion starting from $Z^{m}(0) = x\vee m$ with lower reflecting barrier at $m$. The next lemma states that $Z^{m}$ dominates all inventory processes that have an order-up-to bound at $m$.
\begin{lemma}
\label{lem:Z-dominance}
For a positive integer $m$, let $Z$ be the inventory process given by \eqref{eq:Z} with $Y\in\mathcal{U}_{m}$ and $Z^{m}$ the inventory process given by \eqref{eq:Z-m} with $X$ defined by \eqref{eq:X}. Then, $Z(t)\leq Z^{m}(t)$ on each sample path for all $t\geq 0$.
\end{lemma}

The marginal distribution of $Z^{m}$ can be specified as follows. Let
\[
\psi_{x}^{m}(v,t)=\mathbb{P}[Z^{m}(t) > v \,|\, X(0)=x]\quad \mbox{for }t\geq 0\mbox{ and
}v\geq 0.
\]
Then by (3.63) in \citet{Harrison13}, $\psi_{x}^{m}(v,t) = 1$ for $0\leq v< m$ and
\begin{equation}
\psi_{x}^{m}(v,t) = \Phi\Big(\frac{-v+(x\vee m)-\mu t}{\sigma t^{1/2}}\Big) +\mathrm{e}^{-\lambda (v-m)} \Phi\Big(\frac{-v-(x\vee m)+\mu t}{\sigma t^{1/2}}\Big) \quad\mbox{for } v\geq m,
\label{eq:psi}
\end{equation}
where $\Phi$ is the standard Gaussian cumulative distribution function. Because $Z^{m}$ dominates all inventory processes that have an order-up-to bound at $m$, we may use its marginal distribution to establish boundedness results for policies in $\bar{\mathcal{U}}$.

\begin{lemma}
\label{lem:Exf} 
Let $f:\mathbb{R}\rightarrow\mathbb{R}$ be a differentiable function and $Z$ the inventory process given by \eqref{eq:Z} with $Y\in\bar{\mathcal{U}}$. Assume that there exists a positive integer $ d $ and two positive numbers $a_{0}$ and $a_{1}$ such that
\begin{equation}
|f'(z)|<a_{0}+a_{1}|z|^{d}\quad\mbox{for }z\in\mathbb{R}.
\label{eq:fprime}
\end{equation}
Then,
\begin{equation}
\mathbb{E}_{x}\big[\big|f(Z(t))\big|\big]< \infty\quad\mbox{for }t\geq 0,
\label{eq:ExfZ1}
\end{equation}
and
\begin{equation}
\mathbb{E}_{x}\Big[\int_{0}^{t}f'(Z(u))^2\,\mathrm{d}u\Big] < \infty \quad\mbox{for }t\geq 0.
\label{eq:ExfZ2}
\end{equation}
Moreover,
\begin{equation}
\lim_{t\rightarrow\infty}\frac{1}{t}\mathbb{E}_{x}\big[\big|f(Z(t)) \cdot 1_{[0,\infty)}(Z(t))\big|\big] = 0.
\label{eq:limsup-fz}
\end{equation}
\end{lemma}

The proof of Proposition~\ref{prop:lower-bound} relies on Theorem~\ref{thm:dense} and Lemma~\ref{lem:Exf}, which enable us to establish a lower bound for long-run average costs by examining policies in $ \bar{\mathcal{U}} $, instead of all admissible policies.

\begin{proof}[Proof of Proposition~\ref{prop:lower-bound}]
By Theorem~\ref{thm:dense}, it suffices to consider $Y\in\bar{\mathcal{U}}$, in which case \eqref{eq:ExfZ1}--\eqref{eq:limsup-fz} hold. By \eqref{eq:Z}, \eqref{eq:continuous-part}, and It\^{o}'s formula (see, e.g., Lemma~3.1 in \citealt{DaiYao13a}),
\[
f(Z(t))  =f(x)+\int_{0}^{t}\Gamma f(Z(u))\,\mathrm{d}u+\sigma \int_{0}^{t}f^{\prime }(Z(u))\,\mathrm{d}B(u)+\int_{0}^{t}f^{\prime }(Z(u))\, \mathrm{d}Y^{c}(u)+\sum_{0\leq u\leq t}\Delta f(Z(u))
\]
for $t\geq 0$. Then by \eqref{eq:lower-bound},
\begin{align}
f(Z(t)) & \geq  f(x)+\underaccent{\bar}{\nu} t - \int_{0}^{t}h(Z(u))\,\mathrm{d}u+\sigma \int_{0}^{t}f^{\prime }(Z(u))\,\mathrm{d}B(u)\nonumber\\
& \quad +\int_{0}^{t}f^{\prime }(Z(u))\, \mathrm{d}Y^{c}(u)+\sum_{0\leq u\leq t} \Delta f(Z(u)).
\label{eq:f-bound-1}
\end{align}
By \eqref{eq:ell} and \eqref{eq:f-d}, $ f^{\prime}(z) \geq -(k+\ell) $ for $ z \in \mathbb{R} $, where $ \ell = \infty $ if $ K(0+) > 0 $. Then by \eqref{eq:f-d} and \eqref{eq:f-bound-1},
\begin{align*}
f(Z(t)) & \geq  f(x)+\underaccent{\bar}{\nu} t - \int_{0}^{t}h(Z(u))\,\mathrm{d}u+\sigma \int_{0}^{t}f^{\prime }(Z(u))\,\mathrm{d}B(u)\nonumber\\
& \quad - (k+\ell) Y^{c}(t) - \sum_{0\leq u\leq t} \big( K( \Delta Z(u) ) + k\Delta Z(u) \big).
\end{align*}
Since $\Delta Z(t)=\Delta Y(t)$ and $ Y(t)=\sum_{0\leq u\leq t} \Delta Y(u) + Y^{c}(t)$, the above inequality can be written as
\begin{equation}
f(Z(t)) + \int_{0}^{t}h(Z(u))\,\mathrm{d}u + \sum_{0\leq u\leq t}K(\Delta Y(u))+kY(t)+\ell Y^{c}(t) \geq f(x)+\underaccent{\bar}{\nu} t +\sigma \int_{0}^{t}f^{\prime}(Z(u))\,\mathrm{d}B(u).
\label{eq:cost-bound}
\end{equation}
By \eqref{eq:ExfZ2} and Theorem~3.2.1 in \citet{Oksendal03},
\[
\mathbb{E}_{x}\Big[\int_{0}^{t}f^{\prime }(Z(u))\,\mathrm{d}B(u)\Big]=0.
\]
Since \eqref{eq:ExfZ1} holds, we can take expectation on both sides of
\eqref{eq:cost-bound}, which yields
\[
\mathbb{E}_{x}\big[f(Z(t))\big] + \mathbb{E}_{x}\Big[\int_{0}^{t}h(Z(u))\,\mathrm{d}u +
\sum_{0\leq u\leq t}K(\Delta Y(u))+kY(t)+\ell Y^{c}(t)\Big]
\geq f(x)+\underaccent{\bar}{\nu} t.
\]
Dividing both sides by $t$ and letting $t$ go to infinity, we have
\[
\liminf_{t \rightarrow  \infty} \frac{1}{t} \mathbb{E}_{x}\big[f(Z(t))\big] + \liminf_{t\rightarrow \infty} \frac{1}{t} \mathbb{E}_{x}\Big[\int_{0}^{t}h(Z(u))\, \mathrm{d}u + \sum_{0\leq u\leq t}K(\Delta Y(u))+kY(t)+\ell Y^c(t)\Big] \geq \underaccent{\bar}{\nu}.
\]
Then, it follows from \eqref{eq:AC} that
\begin{equation}
\liminf_{t\rightarrow\infty}\frac{1}{t}\mathbb{E}_{x}\big[f(Z(t))\big]+
\operatorname{AC}(x,Y) \geq \underaccent{\bar}{\nu}.
\label{eq:AC-bound}
\end{equation}
By \eqref{eq:AC-bound}, $\operatorname{AC}(x,Y) \geq \underaccent{\bar}{\nu}$ holds when
\begin{equation}
\liminf_{t\rightarrow \infty}\frac{1}{t}\mathbb{E}_{x}\big[f(Z(t))\big] \leq 0.
\label{eq:ExfZ-less-zero}
\end{equation}
Otherwise, there exists $c>0$ such that
\begin{equation}
\liminf_{t\rightarrow \infty}\frac{1}{t} \mathbb{E}_{x} \big[f(Z(t))\big] > c.
\label{eq:Exfc}
\end{equation}
We next show that $\operatorname{AC}(x,Y)=\infty$ if inequality \eqref{eq:Exfc} holds. Hence, $\operatorname{AC}(x,Y) \geq \underaccent{\bar}{\nu}$ must be true.
	
It follows from \eqref{eq:limsup-fz} and \eqref{eq:Exfc} that
\[
\liminf_{t\rightarrow \infty}\frac{1}{t}\mathbb{E}_{x}\big[f(Z(t)) \cdot 1_{(-\infty, 0)}(Z(t))\big] > c,
\]
and thus
\[
\mathbb{E}_{x}\big[f(Z(t)) \cdot 1_{(-\infty, 0)}(Z(t))\big] > \frac{ct}{2}
\]
for $t$ sufficiently large. By \eqref{eq:fprime-n}, there is some $c_{0}>0$ such that $|f(z)|\leq a_{0}|z|+c_{0}$ for $z<0$. Then, 
\begin{equation}
\mathbb{E}_{x}[|Z(t)|]>\frac{ct-2c_{0}}{2a_{0}}
\label{eq:ExZ}
\end{equation}
for $t$ sufficiently large. Since $h$ is convex with $h'(z)>0$ for $z>0$ and $h'(z)<0$ for $z<0$, we can find $c_{1}, c_{2}>0$ such that $h(z)\geq c_{1}|z|-c_{2}$ for all $z\in\mathbb{R}$. Therefore,
\[
\limsup_{t\rightarrow\infty}\frac{1}{t}\mathbb{E}_{x}\Big[ \int_{0} ^{t}h(Z(u))\,\mathrm{d}u\Big] \geq \limsup_{t\rightarrow\infty} \frac{c_{1}}{t}\mathbb{E}_{x}\Big[ \int_{0}^{t}|Z(u)|\,\mathrm{d}u\Big] -c_{2},
\]
where, by \eqref{eq:ExZ}, the right side must be infinite. Hence, we must have $\operatorname{AC}(x,Y)=\infty$. 
\end{proof}

\begin{remark}
The boundedness conditions \eqref{eq:ExfZ2} and \eqref{eq:limsup-fz} are essential to prove Proposition~\ref{prop:lower-bound} by It\^{o}'s formula. More specifically, condition \eqref{eq:ExfZ2} ensures that \eqref{eq:AC-bound} holds, condition \eqref{eq:limsup-fz} ensures that \eqref{eq:ExfZ-less-zero} holds as long as the long-run average cost is finite, and the lower bound result follows from these two inequalities. Since conditions \eqref{eq:ExfZ2} and \eqref{eq:limsup-fz} do not hold for all admissible policies, Theorem~\ref{thm:dense} is the critical tool for establishing a lower bound for all of them. In the Brownian model studied by \citet{HarrisonETAL83}, \citet{OrmeciETAL08}, and \citet{DaiYao13a,DaiYao13b}, inventory is allowed to be adjusted both upwards and downwards. The optimal policy in that setting is a control band policy under which the inventory level is confined within a finite interval. Because the relative value function under a control band policy is Lipschitz continuous, these authors imposed a Lipschitz assumption on $f$ in their verification theorems. This assumption ensures that condition \eqref{eq:ExfZ2} holds for all admissible policies (which yields \eqref{eq:AC-bound}) and that condition \eqref{eq:ExfZ-less-zero} holds when the long-run average cost is finite, so one can obtain a lower bound for all admissible policies immediately. In our Brownian model, however, only upward adjustments are allowed and the optimal policy is an $(s,S)$ policy whose relative value function is not Lipschitz continuous (see Remark~\ref{rmk:Lipschitz} in \S\ref{sec:proof-1}). Without the Lipschitz assumption, conditions \eqref{eq:ExfZ2} and \eqref{eq:ExfZ-less-zero} may no longer hold for a general admissible policy, even if we assume the associated long-run average cost to be finite. In this case, \citet{WuChao14} and \citet{YaoETAL15} restricted their scope to the subset of policies that satisfy \eqref{eq:ExfZ2} and \eqref{eq:ExfZ-less-zero}. Their lower bounds are established within this subset, and consequently, their proposed policies are proved optimal within the same subset. Theorem~\ref{thm:dense} in the present paper enables us to establish a lower bound for all admissible policies. We can thus prove the proposed policy to be globally optimal.
\end{remark}

\section{Relative value function and long-run average cost}
\label{sec:sS}

In order to prove the proposed policy to be optimal, let us first analyze the long-run average cost under an arbitrary $(s,S)$ policy. An important notion for the analysis is the relative value function under the $(s,S)$ policy. The relative value function and the associated long-run average cost jointly satisfy an ordinary differential equation with some boundary conditions.

\begin{proposition}
\label{prop:sS}
Assume that $h$ satisfies \textnormal{(\ref{itm:H1})--(\ref{itm:H5})}. For any pair $ (s,S) $ with $ s \leq S $, there exists a positive number $\nu$ and a twice continuously differentiable function $V:\mathbb{R}\rightarrow \mathbb{R}$ that jointly satisfy
\begin{equation}
\Gamma V(z)+h(z) = \nu\quad\mbox{for }z\in\mathbb{R}
\label{eq:gamma-V}
\end{equation}
with boundary conditions
\begin{equation}
\begin{cases}
V(S)-V(s)  = -K(S-s)-k(S-s) & \mbox{if }s<S,\\
V'(s)  = -(k + \ell) & \mbox{if }s=S,
\end{cases}
\label{eq:V-s-S}
\end{equation}
and
\begin{equation}
\lim_{z\rightarrow \infty} \mathrm{e}^{-\alpha z}V'(z) = 0 \quad\mbox{for $ \alpha>0 $.}
\label{eq:V-prime-limit}
\end{equation}
Specifically, the solution to \eqref{eq:gamma-V}--\eqref{eq:V-prime-limit} is $\nu = \gamma(s,S)$, where $ \gamma(s,S) $ is given by \eqref{eq:gamma}, and
\begin{equation}
V(z)= -\frac{(z-s)\nu}{\mu}+\frac{\lambda}{\mu}\int_{s}^{z}\int _{0}^{\infty}h(u+y)\mathrm{e}^{-\lambda u}\,\mathrm{d}u\,\mathrm{d}y,
\label{eq:V}
\end{equation}
where $V$ is unique up to addition by a constant. Assume that $ K $ satisfies \textnormal{(\ref{itm:S3})} if $ s = S $. Then, $\operatorname{AC} (x,U(s,S))=\gamma(s,S)$, i.e., $\gamma(s,S)$ is the long-run average cost under the $(s,S)$ policy.
\end{proposition}

\begin{remark}
\label{rmk:relative}
For $z \geq s$, $V(z)$ can be interpreted as the cost disadvantage of inventory level $z$ relative to the reorder level $s$. Under the $(s,S)$ policy, $T(s)$ defined by \eqref{eq:tau} can be interpreted as the first ordering time, given that $ Z(0-) = z $, and $H_{z}(s)$ is the expected holding and shortage cost during $[0, T(s)]$. Following the arguments in Remark~\ref{rmk:sS}, we have
\[
\mathbb{E}_{z}[T(s)]=\frac{z-s}{\mu} \quad\mbox{ and }\quad H_{z}(s) = \frac{\lambda}{\mu}\int_{s}^{z} \int_{0}^{\infty}h(u+y)\mathrm{e}^{-\lambda u}\,\mathrm{d}u\,\mathrm{d}y.
\]
By \eqref{eq:V}, $V(z)$ can be decomposed into
\[
V(z) = H_{z}(s) -\nu\cdot\mathbb{E}_{z}[T(s)].
\]
In this equation, $H_{z}(s)$ is the cost disadvantage of a system starting from time zero with initial level $ Z(0-) = z $ compared with a system starting from time $T(s)$ with initial level $Z(T(s)-)=s$, while $\nu \cdot \mathbb{E}_{z} [T(s)]$ represents the cost disadvantage of a system starting from time zero compared with the delayed system starting from time $T(s)$. As the difference between these two costs, $V(z)$ represents the relative cost disadvantage of inventory level $z$ compared with the reorder level $s$.
\end{remark}

\begin{proof}[Proof of Proposition~\ref{prop:sS}]
We obtain the explicit solution $ (\nu, V) $ to the boundary value problem \eqref{eq:gamma-V}--\eqref{eq:V-prime-limit} as follows. If such a solution exists, write $g(z) = V^{\prime}(z)$ for $z\in\mathbb{R}$. By \eqref{eq:gamma-V} and \eqref{eq:V-prime-limit}, $g$ satisfies the following linear first-order ordinary differential equation,
\[
g^{\prime}(z)-\lambda g(z)=-\frac{\lambda h(z)}{\mu} + \frac{\lambda\nu}{\mu}
\]
with boundary condition
\[
\lim_{z\rightarrow\infty}\mathrm{e}^{-\alpha z}g(z)=0\quad\text{for }\alpha>0.
\]
Since $h$ is polynomially bounded, for each $ \nu \in \mathbb{R} $, the above equation has a unique solution
\[
g(z)=-\frac{\nu}{\mu}+\frac{\lambda}{\mu}\int_{0}^{\infty} h(y+z)\mathrm{e}^{-\lambda y}\,\mathrm{d}y,
\]
which yields \eqref{eq:V}. By the boundary condition \eqref{eq:V-s-S}, we obtain $\nu = \gamma(s,S)$.
	
It remains to show $\operatorname{AC}(x,U(s,S))=\gamma(s,S)$. By \eqref{eq:Z}, \eqref{eq:continuous-part}, \eqref{eq:gamma-V}, and It\^{o}'s formula,
\begin{align}
V(Z(t))& = V(Z(0)) + \nu t - \int_{0}^{t} h(Z(u))\,\mathrm{d}u + \sigma\int_{0}^{t} g(Z(u))\,\mathrm{d}B(u)\nonumber\\
&\quad + \int_{0}^{t}g(Z(u))\,\mathrm{d}Y^{c} (u)+\sum_{0< u\leq t} \Delta V(Z(u)).
\label{eq:Ito-V}
\end{align}
Under the $(s,S)$ policy with $s<S$, it follows from \eqref{eq:V-s-S} that $\Delta V(Z(u)) = -K(S-s)-k(S-s)$ whenever $\Delta Z(u)>0$ and $u>0$. Since $Y^{c}(t)=0$ for $t\geq0$, equation \eqref{eq:Ito-V} turns out to be
\[
V(Z(t))=V(Z(0)) + \nu t - \int_{0}^{t} h(Z(u))\,\mathrm{d}u+\sigma\int_{0}^{t} g(Z(u))\,\mathrm{d}B(u)-K(S-s)\tilde{\mathcal{J}}(t)-k(Y(t)-Y(0)),
\]
where $\tilde{\mathcal{J}}(t)$ is the cardinality of $\{u\in(0,t]:\Delta Y(u)>0\}$. By \eqref{eq:ExfZ2},
\begin{equation}
\mathbb{E}_{x}[V(Z(t))] = \mathbb{E}_{x}[V(Z(0))+kY(0)]+\nu t-\mathbb{E}_{x}\Big[ \int_{0}^{t}h(Z(u))\,\mathrm{d}u + K(S-s)\tilde{\mathcal{J}}(t)+kY(t)\Big].
\label{eq:VZ-sS}
\end{equation}
When $s=S$, by Lemma~\ref{lem:Skorokhod-Map}, both $Y$ and $Z$ have continuous sample paths. By \eqref{eq:V-s-S} and \eqref{eq:Ito-V},
\[
V(Z(t)) = V(Z(0))+\nu t - \int_{0}^{t} h(Z(u)) \,\mathrm{d}u + \sigma\int_{0}^{t} g(Z(u))\,\mathrm{d}B(u) - (k+\ell)Y^c(t).
\]
Since $ Y^{c}(t) = Y(t) - Y(0) $, taking expectation on both sides, we obtain
\begin{equation}
\mathbb{E}_{x}[V(Z(t))] = \mathbb{E}_{x}[V(Z(0)) + kY(0)] + \nu t - \mathbb{E}_{x}\Big[ \int_{0}^{t}h(Z(u))\,\mathrm{d}u + kY(t) + \ell Y^{c}(t)\Big].
\label{eq:VZ-s}
\end{equation}
	
Under the $(s,S)$ policy with $s\leq S$, $Y(0)\leq |S-x|$ and $x\wedge S\leq Z(0)\leq x\vee S$. It follows that
\[
\lim_{t\rightarrow\infty}\frac{1}{t}\mathbb{E}_{x}[V(Z(0))+kY(0)]=0.
\]
Because $|V(Z(t))| \leq |V(Z(t)) \cdot 1_{[0,\infty)}(Z(t))| + \max\{|V(z)|:(s\wedge0) \leq z\leq 0\}$, we obtain
\[
\lim_{t\rightarrow\infty}\frac{1}{t}\mathbb{E}_{x}[V(Z(t))]=0
\]
by \eqref{eq:limsup-fz}. Then, it follows from \eqref{eq:AC} and \eqref{eq:VZ-sS}--\eqref{eq:VZ-s} that $\operatorname{AC}(x,U(s,S))=\gamma(s,S)$. 
\end{proof}

\begin{remark}
\label{rmk:smooth-pasting}
When the setup cost is constant for any order quantity, the optimal reorder and order-up-to levels can be obtained by adding a smooth pasting condition \begin{equation}
V^{\prime}(s^{\star}) = V^{\prime}(S^\star) = -k,
\label{eq:smooth-pasting}
\end{equation}
where, with slight abuse of notation, $V$ should be understood as the relative value function under $U(s^{\star},S^{\star})$; see \citet{Bather66}, \citet{Taksar85}, and \citet{Sulem86} for the interpretation of the smooth pasting condition. This condition, together with \eqref{eq:gamma-V}--\eqref{eq:V-prime-limit}, defines a free boundary problem by which $(s^{\star},S^{\star})$ can be uniquely determined. In our Brownian model, however, the general setup cost function has imposed a quantity constraint on each setup cost value. With these constraints, the smoothness condition \eqref{eq:smooth-pasting} may no longer hold at the free boundary. In other words, the smooth pasting method cannot be used for our problem.
\end{remark}

\section{Optimal ordering policy}
\label{sec:proof-1}

The optimality result is proved in this section. We first confine ordering policies to the $ (s,S) $ type, proving the existence of the optimal $ (s,S) $ policy. Then, we show that the relative value function associated with the optimal $ (s,S) $ policy and the resulting long-run average cost jointly satisfy the conditions in the verification theorem, thus proving Theorem~\ref{thm:main-results} by the lower bound approach.

We establish a series of lemmas to prove Proposition~\ref{prop:z-s-S-star} and Theorem~\ref{thm:main-results}. In particular, the following function $ g_{0} : \mathbb{R} \rightarrow \mathbb{R}_{+} $, defined by
\begin{equation}
g_{0}(z) = \frac{\lambda}{\mu}\int_{0}^{\infty} h(y+z)\mathrm{e}^{-\lambda y}\,\mathrm{d}y,
\label{eq:g0}
\end{equation}
is frequently used in the analysis. The first derivative of $ g_{0} $ is
\begin{equation}
g_{0}^{\prime}(z)=\frac{\lambda}{\mu}\Big(\lambda\int_{0}^{\infty} h(y+z)\mathrm{e}^{-\lambda y}\,\mathrm{d}y-h(z)\Big),
\label{eq:g0-prime}
\end{equation}
and $ g_{0} $ is a solution to the linear first-order ordinary differential equation
\begin{equation}
\frac{1}{2}\sigma^{2} g_{0}^{\prime}(z) - \mu g_{0}(z) + h(z) = 0.
\label{eq:g0-equation}
\end{equation}
Using the derivative, we specify the monotone intervals of $g_{0}$ in the following lemma.
\begin{lemma}
\label{lem:g0}
Assume that $h$ satisfies \textnormal{(\ref{itm:H1})--(\ref{itm:H5})}. Then,
\begin{equation}
\lim_{z\rightarrow \pm \infty}g_{0}(z) = \infty
\label{eq:g0-limit}
\end{equation}
and
\begin{equation}
\begin{cases}
g'_{0}(z)<0 & \mbox{for }z<z^{\star},\\
g'_{0}(z)=0 & \mbox{for }z=z^{\star},\\
g'_{0}(z)>0 & \mbox{for }z>z^{\star},
\end{cases}
\label{eq:z-star}
\end{equation}
where $z^{\star}$ is uniquely determined by \eqref{eq:optimal-barrier} and is less than zero.
\end{lemma}

\begin{remark}
\label{rmk:Lipschitz}
The relative value function $V$ given by \eqref{eq:V} satisfies $ V'(z) = -\nu/\mu + g_{0}(z) $, so by \eqref{eq:g0-limit}, $V'$ is unbounded. This implies that the relative value function for an $(s,S)$ policy is not Lipschitz continuous.
\end{remark}

If each order quantity is fixed at $ \xi > 0 $, the optimal reorder and order-up-to levels will be determined by minimizing the holding and shortage cost. We use $ \tilde{s}(\xi) $ to denote the optimal reorder level, and the corresponding order-up-to level is $ \tilde{S}(\xi) = \tilde{s}(\xi) + \xi $. If a base stock policy is used, we use $ \tilde{s}(0) $ to denote the optimal base stock level, in which case $ \tilde{S}(0) = \tilde{s}(0) $. For $ \xi \geq 0 $, $ \tilde{s}(\xi) $ is specified in Lemmas~\ref{lem:sQ}--\ref{lem:theta}.

\begin{lemma}
\label{lem:sQ}
Assume that $h$ satisfies \textnormal{(\ref{itm:H1})--(\ref{itm:H5})}. Then for each $ \xi > 0 $, there exists a unique $(\tilde{s}(\xi),\tilde{S}(\xi)) \in \mathbb{R}^{2} $ such that 
\begin{equation}
g_{0}(\tilde{s}(\xi)) = g_{0}(\tilde{S}(\xi)) \quad \mbox{and} \quad \tilde{S}(\xi) = \tilde{s}(\xi) + \xi.
\label{eq:g0-equal}
\end{equation}
This solution satisfies
\begin{equation}
\tilde{s}(\xi) < z^{\star} < \tilde{S}(\xi) \quad\mbox{for $ \xi > 0 $,}
\label{eq:sQ}
\end{equation}
\begin{equation}
\lim_{\xi \rightarrow \infty} \tilde{s}(\xi) = -\infty \quad \mbox{and} \quad \lim_{\xi \rightarrow \infty} \tilde{S}(\xi) = \infty.
\label{eq:s-tilde-infinity}
\end{equation}
Moreover, both $ \tilde{s} $ and $ \tilde{S} $ are differentiable on $ (0,\infty) $, with 
\begin{equation}
\tilde{s}^{\prime}(\xi) < 0 \quad \mbox{and} \quad \tilde{S}^{\prime}(\xi) > 0 \quad \mbox{for $ \xi > 0 $.}
\label{eq:s-tilde-prime}
\end{equation}
\end{lemma}

The value of $ \tilde{s}(\xi) $ is determined by \eqref{eq:g0-equal} for $ \xi > 0 $. Taking $ \tilde{s}(0) = z^{\star}$, we extend the domain of $ \tilde{s} $ to $ [0,\infty) $. For notational convenience, let us write
\begin{equation}
\tilde{\gamma} (s, \xi) = \gamma(s, s+\xi) \quad \mbox{for $ s \in \mathbb{R} $ and $ \xi \geq 0 $,} 
\label{eq:gamma-tilde-definition}
\end{equation}
which is the long-run average cost with the reorder level fixed at $ s $ and the order-up-to level fixed at $ s + \xi $. By \eqref{eq:gamma} and \eqref{eq:g0},
\begin{equation}
\tilde{\gamma} (s, \xi) = 
\begin{dcases}
k\mu + \frac{K(\xi)\mu}{\xi} + \frac{\mu}{\xi} \int_{s}^{s+\xi} g_{0}(y)\,\mathrm{d}y & \mbox{if $\xi > 0$},\\
(k + \ell)\mu + \mu g_{0}(s) &\mbox{if $\xi = 0$.}
\end{dcases}
\label{eq:gamma-tilde}
\end{equation}
For $ \xi \geq 0 $, let
\begin{equation}
\theta(\xi) = \inf \{ \tilde{\gamma}(s,\xi) : s \in \mathbb{R} \},
\label{eq:theta-definition}
\end{equation}
which is the minimum long-run average cost when the quantity of each order is fixed at $ \xi $ (a base stock policy is used if $ \xi = 0 $). The next lemma says that this minimum cost can be attained by setting the reorder level at $ \tilde{s}(\xi) $. In addition, $ \theta $ is lower semicontinuous.
\begin{lemma}
\label{lem:theta}
Assume that $ K $ satisfies \textnormal{(\ref{itm:S1})}, \textnormal{(\ref{itm:S3})}, and \textnormal{(\ref{itm:S4})}, and that $ h $ satisfies \textnormal{(\ref{itm:H1})--(\ref{itm:H5})}. Then, 
\begin{equation}
\theta(\xi) = \tilde{\gamma}(\tilde{s}(\xi),\xi) \quad \mbox{for $ \xi \geq 0 $},
\label{eq:theta}
\end{equation}
where $ \tilde{s}(0) = z^{\star} $ and $ \tilde{s}(\xi) $ is determined by \eqref{eq:g0-equal} for $ \xi > 0 $. Moreover, $ \theta $ is lower semicontinuous on $ [0,\infty) $ and satisfies \begin{equation}
\lim_{\xi \rightarrow \infty} \theta (\xi) = \infty .
\label{eq:theta-infinity}
\end{equation}
\end{lemma}

\begin{proof}[Proof of Proposition~\ref{prop:z-s-S-star}]
By \eqref{eq:gamma-tilde-definition} and \eqref{eq:theta-definition},
\[ 
\inf \{ \gamma(s,S) : s \leq S \} = \inf \{ \theta(\xi) : \xi \geq 0 \}.
\]
To prove \eqref{eq:s*S*}, we need to show that there exists $ \xi^{\star} \geq 0 $ such that $ \theta(\xi^{\star}) = \inf \{ \theta(\xi) : \xi \geq 0 \}$. 
	
By \eqref{eq:theta-infinity}, if $ \xi^{\star} $ exists, there must be some $ M < \infty $ such that $ \xi^{\star} \leq M$. Because $ \theta $ is lower semicontinuous on $ [0,\infty) $, by the extreme value theorem (see, e.g., Theorem~B.2 in \citealt{Puterman94}), there exists $ \hat{\xi} \in [0,M] $ such that $ \theta(\hat{\xi}) \leq \theta(\xi) $ for all $ \xi \in [0,M] $. Hence, $ \xi^{\star} = \hat{\xi} $ must be a minimizer of $ \theta $. Taking $ s^{\star} = \tilde{s}(\xi^{\star})$ and $ S^{\star} = \tilde{S}(\xi^{\star})$, we deduce that \eqref{eq:s*S*} holds.
	
Lemma~\ref{lem:g0} provides the properties of $ z^{\star} $. If $ K(0+) > 0 $, we obtain $ \theta(0) = \tilde{\gamma}(z^{\star},0) = \infty $ because $ \ell = \infty $. This implies that $ \xi^{\star} > 0 $, and by \eqref{eq:sQ}, we obtain $ s^{\star} < z^{\star} < S^{\star}$. If $ K(0+) = 0 $, it may happen that $ \xi^{\star} = 0 $, in which case $ s^{\star} = z^{\star} = S^{\star}$ since $ \tilde{s}(0) = z^{\star} $. If $ K(0+) = 0 $ and $ \xi^{\star} > 0 $, we have $ s^{\star} < z^{\star} < S^{\star}$ again by \eqref{eq:sQ}. 
\end{proof}

It remains to prove the global optimality of $ U(s^{\star},S^{\star}) $ using the verification theorem. Under this policy, the long-run average cost in \eqref{eq:gamma} and the relative value function in \eqref{eq:V} satisfy all conditions specified in Proposition~\ref{prop:lower-bound} except for \eqref{eq:fprime-n}. The relative value function should thus be modified to fulfill this condition. To this end, we establish the following lemma.

\begin{lemma}
\label{lem:bar-gamma}
Assume that $K$ satisfies \textnormal{(\ref{itm:S1})--(\ref{itm:S4})} and that $h$ satisfies \textnormal{(\ref{itm:H1})--(\ref{itm:H5})}. Then, there exists $\underaccent{\bar}{s} \in (-\infty, z^{\star}) $ such that 
\begin{equation}
\bar{\gamma}(s,\xi) \geq \gamma(s^{\star},S^{\star}) \quad\mbox{for $ s\in \mathbb{R} $ and $\xi > 0$},
\label{eq:gamma-tilde-bound}
\end{equation}
where
\begin{equation}
\bar{\gamma}(s,\xi) = k\mu +\frac{K(\xi)\mu}{\xi} + \frac{\mu}{\xi} \int_{s}^{s+\xi} g_0(y \vee \underaccent{\bar}{s})\,\mathrm{d}y.
\label{eq:gamma-bar}
\end{equation}
\end{lemma}

The modified relative value function is defined by
\begin{equation}
V^{\star}(z) = - \frac{\gamma(s^{\star},S^{\star})}{\mu}(z - \underaccent{\bar}{s}) + \int_{\underaccent{\bar}{s}} ^ {z} g_{0}(y \vee \underaccent{\bar}{s}) \,\mathrm{d}y \quad \mbox{for $z \in \mathbb{R}$},
\label{eq:V-star}
\end{equation}
with which we are ready to present the proof of Theorem~\ref{thm:main-results}.
\begin{proof}[Proof of Theorem~\ref{thm:main-results}]
Let us show that $ \big(\gamma(s^{\star},S^{\star}), V^{\star}\big) $ satisfies all conditions specified in Proposition~\ref{prop:lower-bound}, so $ U(s^{\star}, S^{\star}) $ is an optimal ordering policy. Clearly, $ V^{\star} $ is twice differentiable except at $ \underaccent{\bar}{s} $. By \eqref{eq:gamma-bar}--\eqref{eq:V-star},
\[
\bar{\gamma}(s,\xi)=\frac{\mu}{\xi}\big(  k\xi + K(\xi) + V^{\star}(s+\xi) - V^{\star}(s)\big) + \gamma(s^{\star},S^{\star}) \quad \mbox{for $ s \in \mathbb{R} $ and $ \xi > 0 $.}
\]
Then by \eqref{eq:gamma-tilde-bound},
\[  
V^{\star}(s+\xi) - V^{\star}(s) \geq -K(\xi) - k\xi,
\]
which implies that $ V^{\star} $ satisfies \eqref{eq:f-d}. By \eqref{eq:z-star}, $ g_{0}(z^{\star}) \leq g_{0}(z \vee \underaccent{\bar}{s}) \leq \big(g_{0}(\underaccent{\bar}{s}) \vee g_{0}(0) \big) $ for $ z < 0 $, from which condition \eqref{eq:fprime-n} follows. Condition \eqref{eq:fprime-p} holds because $ h $ is polynomially bounded. For $ z > \underaccent{\bar}{s} $, it follows from \eqref{eq:g0-equation} and \eqref{eq:V-star} that
\[  
\Gamma V^{\star} (z) + h(z) = \frac{1}{2}\sigma^{2} g_{0}^{\prime}(z) - \mu g_{0}(z) + h(z) + \gamma (s^{\star}, S^{\star}) = \gamma (s^{\star}, S^{\star}).
\]
For $ z < \underaccent{\bar}{s} $, $ g_{0}^{\prime}(z \vee \underaccent{\bar}{s}) = 0 $. Since $ \underaccent{\bar}{s} < z^{\star} $, it follows from \eqref{eq:z-star} that $ g_{0}^{\prime} (z) < 0 $ and $ g_{0}(z) > g_{0}(\underaccent{\bar}{s}) $. Then,
\[  
\Gamma V^{\star} (z) + h(z) = - \mu g_{0}(\underaccent{\bar}{s}) + h(z) + \gamma (s^{\star}, S^{\star}) > \frac{1}{2}\sigma^{2} g_{0}^{\prime}(z) - \mu g_{0}(z) + h(z) + \gamma (s^{\star}, S^{\star}) = \gamma (s^{\star}, S^{\star}).
\]
Hence,  $ \big(\gamma(s^{\star},S^{\star}), V^{\star}\big) $ satisfies condition \eqref{eq:lower-bound}. 
\end{proof}

\section{Policies subject to order-up-to bounds}
\label{sec:proof-3}

This section is devoted to the proof of Theorem~\ref{thm:dense}. Let $Y$ be an admissible policy. We first modify this policy to construct a policy $Y_{m}\in\mathcal{U}_{m}$, where $m$ is a fixed positive integer. Then, we prove that $\{Y_{m}\in\mathcal{U}_{m}:m=1,2,\ldots\}$ has a subsequence that satisfies \eqref{eq:dense}.

For each $Y$, we would construct a policy $Y_{m}\in\mathcal{U}_{m}$ that incurs less holding and shortage cost and less proportional cost. As $m$ goes large, the average setup cost under $Y_{m}$ should be asymptotically dominated by that under $Y$. Although by imposing an order-up-to bound, we can easily construct a policy that maintains a lower inventory level, we must make additional adjustments to ensure that the shortage level under $Y_{m}$ will not be higher. Such a policy is constructed as follows.

Let $Y_{m}^{c}$ be the continuous part of $Y_{m}$. Under $Y_{m}$, the inventory level at time $t$ is
\begin{equation}
Z_{m}(t)=X(t)+Y_{m}(t),
\label{eq:Zm}
\end{equation}
where $X(t)$ is given by \eqref{eq:X} and
\[
Y_{m}(t)=Y_{m}^{c}(t)+\sum_{0\leq u\leq t}\Delta Y_{m}(u).
\]
The continuous part of $Y_{m}$ is constructed by
\begin{equation}
Y_{m}^{c}(t)=\int_{0}^{t}1_{(-\infty, m]}(Z_{m}(u-))\,\mathrm{d}Y^{c}(u),
\label{eq:Ymc}
\end{equation}
where $Y^{c}$ is the continuous part of $Y$. On each sample path, $Y_{m}$ may have a jump either at a jump time of $Y$ or at a hitting time of zero by $Z_{m}$. Let $J_{m}=$ $\{t\geq0:\Delta Y_{m}(t)>0\}$ be the set of jump times of $Y_{m}$, $J=\{t\geq0:\Delta Y(t)>0\}$ the set of jump times of $Y$, and $I_{m}=\{t\geq0:Z_{m}(t-)=0\}$ the set of hitting times of zero by $Z_{m}$. Then, $J_{m}\subset J \cup I_{m}$. The size of each jump of $Y_{m}$ is specified as follows:
\begin{enumerate}[noitemsep,label=(J\arabic*),ref=J\arabic*]
\item \label{itm:J1} $\Delta Y_{m}(t)=0$ for $t\in J$, if $Z_{m}(t-)>m/2$; 
\item \label{itm:J2} $\Delta Y_{m}(t)=\Delta Y(t)$ for $t\in J$, if $Z_{m}(t-)\leq m/2$ and $Z_{m}(t-)+\Delta Y(t)\leq m$;
\item \label{itm:J3} $\Delta Y_{m}(t)=m-Z_{m}(t-)$ for $t\in J$, if $Z_{m}(t-)\leq m/2$ and $Z_{m}(t-)+\Delta Y(t)>m$;
\item \label{itm:J4} $\Delta Y_{m}(t)= (Z(t)\wedge m)^{+}$ for $t\in I_{m}\setminus J$, where $Z$ is the inventory process under policy $Y$.
\end{enumerate}
In other words, $Y_{m}$ does not make jumps when the inventory level is above $m/2$. If the inventory level is below $m/2$, $Y_{m}$ has simultaneous jumps with $Y$. Each simultaneous jump takes the corresponding jump size of $Y$, as long as the inventory level will not exceed $m$ after that jump; otherwise, the simultaneous jump will replenish the inventory level to $m$. In addition, $Y_{m}$ may have jumps when the inventory level reaches zero. In this case, it will replenish the inventory level to $Z(t)\wedge m$, if the inventory level of the system under policy $Y$ satisfies $Z(t)>0$; otherwise, $Y_{m}$ does not make a jump.

The above policy construction procedure is illustrated in Figure~\ref{fig:Zm}. We plot a sample path of the inventory process under policy $Y$ and the corresponding sample path under policy $Y_{m}$. We use the blue curve for the inventory process under $Y$, the dashed red curve for that under $Y_{m}$, and the black curve for their identical parts. The type of each jump is indicated beside the jump point. In addition to these jumps, we assume that $Y^{c}$ increases over time intervals (C1) and (C2). As $Z_{m}$ is below $m$ over (C1), $Y^{c}$ and $Y^{c}_{m}$ have the same increments during the time; however, $Y^{c}_{m}$ does not increase over (C2) while $Z_{m}$ is above $m$. 

\begin{figure}[t]
\centering
\includegraphics[width=4.5in]{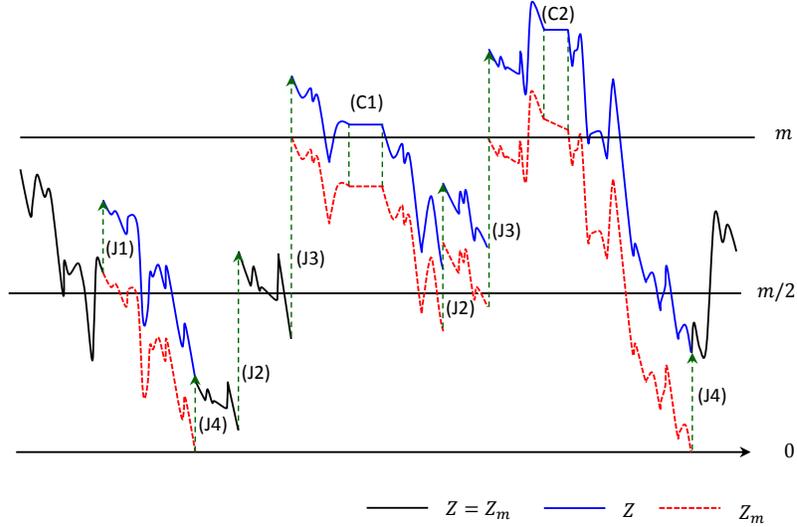}
\caption{A sample path of the inventory process under policy $Y$ and the corresponding sample path under the modified policy $Y_{m}$, with $Z$ given by \eqref{eq:Z} and $Z_{m}$ given by \eqref{eq:Zm}.}
\label{fig:Zm}
\end{figure}

The following lemma states that compared with policy $Y$, the modified policy $Y_{m}$ maintains a lower inventory level and the same shortage level.

\begin{lemma}
\label{lem:Zm}
Let $Y$ be an admissible policy. For a fixed positive integer $m$, let $Y_{m}$ be the policy constructed according to \eqref{eq:Ymc} and \textnormal{(\ref{itm:J1})--(\ref{itm:J4})}. Then, $Z_{m}(t)\leq Z(t)$ for all $t\geq0$ on each sample path, where $Z$ is the inventory process under policy $Y$. In particular, $Z_{m}(t)=Z(t)$ when $Z_{m}(t)<0$.
\end{lemma}

Lemma~\ref{lem:Zm} implies that the modified policy $Y_{m}$ incurs less holding and shortage cost and less proportional cost than policy $Y$. To prove the comparison theorem, we should also establish asymptotic dominance between the average setup costs incurred by these two policies.

\begin{proof}[Proof of Theorem~\ref{thm:dense}]
Because $h$ is nondecreasing on $[0,\infty)$, it follows from Lemma~\ref{lem:Zm} that
\[
\int_{0}^{t}h(Z_{m}(u)) \, \mathrm{d}u \leq \int_{0}^{t} h(Z(u)) \,\mathrm{d}u \quad\mbox{for $ t\geq 0 $,}
\]
i.e., $Y_{m}$ incurs less holding and shortage cost than $Y$. Since $Z_{m}(t)\leq Z(t)$, we have $Y_{m}(t)\leq Y(t)$, so $Y_{m}$ incurs less proportional cost, too. By \eqref{eq:Ymc}, $ Y^{c}_{m}(t) \leq Y^{c}(t) $ for all $ t \geq 0 $. Then, if $ K(0+) = 0 $, we obtain $ \ell Y^{c}_{m}(t) \leq \ell Y^{c}(t) $, i.e., $Y_{m}^{c}$ incurs less setup cost than $Y^{c}$. If $ K(0+) > 0 $ and there exists some $t>0$ for which $Y^{c}(t)>0$, the cumulative setup cost incurred by $Y^{c}$ must be infinite. Therefore, we only need to consider the setup costs incurred by jumps.
	
When a jump of type (\ref{itm:J2}) is made by $Y_{m}$, the setup cost is equal to the cost incurred by the simultaneous jump of $Y$.
	
Consider two consecutive jumps of type (\ref{itm:J3}). Let $t_{1}$ and $t_{2}$ be their respective jump times with $0\leq t_{1} < t_{2}$. Because $X$ has continuous sample paths and $Y_{m}$ is nondecreasing, it follows from \eqref{eq:Zm} that $X(t_{1})-X(t_{2})\geq Z_{m}(t_{1})-Z_{m}(t_{2}-)\geq m/2$. Let
\[
t_{3}=\inf\Big\{u\in(0,t_{2}-t_{1}]:X(t_{1}+u)=X(t_{1})-\frac{m}{2}\Big\}.
\]
By the strong Markov property of Brownian motion, $t_{3}$ has the same distribution as
\[
\tau=\inf\Big\{u>0:-\mu u+\sigma B(u)=-\frac{m}{2}\Big\}
\]
where $B$ is a standard Brownian motion starting with $ B(0) = 0$. Because $\tau$ is the first hitting time of $-m/2$ by a Brownian motion with drift $-\mu$, we obtain $\mathbb{E}_{x}[t_{3}]=m/(2\mu)$. Let $N_{m,1}(t)$ be the number of jumps of type (\ref{itm:J3}) made by $Z_{m}$ up to time $t$. Because $t_{2}-t_{1}\geq t_{3}$, it follows that 
\[
\mathbb{E}_{x}[N_{m,1}(t)]\leq\frac{2\mu t}{m}+1.
\]
	
Now consider two consecutive positive jumps of type (\ref{itm:J4}). Let $\tilde{t} _{1}$ and $\tilde{t}_{2}$ be their respective jump times with $0\leq\tilde{t}_{1}<\tilde{t}_{2}$. We would like to show that there exists some $\tilde{t}_{0}\in[\tilde{t}_{1},\tilde{t}_{2})$ for which $Z_{m}(\tilde{t}_{0})>m/2$. Since $\Delta Y_{m}(\tilde{t}_{2})>0$, $Z_{m}(\tilde{t}_{2}-)\neq Z(\tilde{t}_{2}-)$. If $Z_{m}(\tilde{t}_{1}) = Z(\tilde{t}_{1})$, $\tilde{t}_{0}$ must exist because otherwise, $Y_{m}$ can only have jumps of type (\ref{itm:J2}) during $(\tilde{t}_{1},\tilde{t}_{2})$ and this yields $Z_{m}(\tilde{t}_{2}-) = Z(\tilde{t}_{2}-)$, a contradiction. If $Z_{m}(\tilde{t}_{1})\neq Z(\tilde{t}_{1})$, we have $Z_{m}(\tilde{t}_{1})=m$ and thus set $\tilde{t}_{0} = \tilde{t}_{1}$. Therefore, $Z_{m}(\tilde{t}_{0})>m/2$ holds for some $\tilde{t}_{0} \in [\tilde{t}_{1},\tilde{t}_{2})$. Let
\[
\tilde{t}_{3}=\inf\Big\{u\in(0,\tilde{t}_{2}-\tilde{t}_{0}] : X(\tilde{t}_{0}+u)=X(\tilde{t}_{0})-\frac{m}{2}\Big\},
\]
which also satisfies $\mathbb{E}_{x}[\tilde{t}_{3}] = m/(2\mu)$. Let $N_{m,2}(t)$ be the number of positive jumps of type (\ref{itm:J4}) made by $Z_{m}$ up to time $t$. Because $\tilde{t}_{3} < \tilde{t}_{2} - \tilde{t}_{1}$, we have
\[
\mathbb{E}_{x}[N_{m,2}(t)]\leq\frac{2\mu t}{m}+1.
\]
	
Put $ \bar{K} = \sup\{ K(\xi) : \xi > 0 \}$, which is finite by (\ref{itm:S2}). By the discussion above, 
\[
\operatorname{AC}(x,Y_{m})-\operatorname{AC}(x,Y)\leq\limsup_{t\rightarrow \infty}\frac{1}{t}{\mathbb{E}_{x}[N_{m,1}(t)+N_{m,2}(t)]\bar{K}} \leq \frac{4\mu \bar{K}}{m},
\]
from which we deduce that 
\[ \limsup_{m\rightarrow\infty}\operatorname{AC}(x,Y_{m})\leq\operatorname{AC}(x,Y). 
\]
By the Bolzano--Weierstrass theorem, $\{\operatorname{AC}(x,Y_{m}) : m=1,2,\ldots\}$ has a convergent subsequence, so inequality \eqref{eq:dense} holds. 
\end{proof}

\section{Optimal ordering policy with a step setup cost function}
\label{sec:algorithm}

By Proposition~\ref{prop:z-s-S-star} and Theorem~\ref{thm:main-results}, we need to solve the following nonlinear optimization problem to obtain the optimal ordering policy,
\begin{equation}
\begin{array}{l@{\quad}l@{}}
\mbox{min} & \gamma(s,S)\\ 
\mbox{s.t.} & s \leq S,
\end{array}
\label{eq:optimization}
\end{equation}
where $ \gamma(s,S) $, given by \eqref{eq:gamma}, is the long-run average cost under the $ (s,S) $ policy. With a setup cost function that satisfies (\ref{itm:S1})--(\ref{itm:S4}), one may solve \eqref{eq:optimization} numerically by a standard grid search or a random search (see, e.g., Chapter~4 in \citealt{HendrixToth10}). When the setup cost function takes certain forms, it is possible to obtain the optimal solution in a more efficient way. In this section, we consider the optimal ordering policy when the setup cost is a step function satisfying (\ref{itm:S1})--(\ref{itm:S4}), i.e.,
\begin{equation}
K(\xi)=\sum_{n=1}^{N}K_{n}\cdot 1_{(Q_{n-1},Q_{n})}(\xi)+\sum_{n=1}^{N-1}(K_{n} \wedge K_{n+1})\cdot1_{\{Q_{n}\}}(\xi) \quad \mbox{for $ \xi \geq 0 $,}
\label{eq:setupcost}
\end{equation}
where $N$ is a positive integer, $0=Q_{0}<Q_{1}<\cdots<Q_{N-1}<Q_{N}=\infty$, and $K_{1},\ldots,K_{N}$ are nonnegative real numbers with $K_{n}\neq K_{n+1}$ for $n=1,\ldots,N-1$. The setup cost is $K_{n}$ for any order quantity within the open interval $(Q_{n-1},Q_{n})$. When the order quantity is $Q_{n}$ for $n=1,\ldots,N-1$, we assume that the buyer is required to pay the lower fee of $K_{n}$ and $K_{n+1}$. This step function encompasses most setup cost structures in the literature and in practice, e.g., those in \eqref{eq:two-piece}--\eqref{eq:free-shipping}.

When the step setup cost function in \eqref{eq:setupcost} has $ K_{1} = 0 $, by placing small orders, the inventory system can be exempt from setup fees without incurring additional holding and shortage cost. In this case, we may assume the setup cost to be a zero function. As we discussed in Remark~\ref{rmk:base-stock}, the optimal policy will be a base stock policy whose base stock level is fixed at $ z^{\star} $. When the setup cost function in \eqref{eq:setupcost} has $ K_{1} > 0 $, by Theorem~\ref{thm:main-results} and Proposition~\ref{prop:z-s-S-star}, the optimal reorder and order-up-to levels must satisfy $ s^{\star} < S^{\star} $. We may follow a five-step procedure to obtain the optimal parameters.

\begin{enumerate}[itemindent=1em,label=Step~\arabic*:,ref=\arabic*]
\item \label{itm:step-1} Obtain $z^{\star}$ by 	solving the integral equation \eqref{eq:optimal-barrier}. If $ K_{1} = 0 $ in \eqref{eq:setupcost}, taking $ s^{\star} = S^{\star} = z^{\star} $, we obtain an optimal ordering policy, which is a base stock policy with base stock level $ z^{\star} $. The minimum long-run average cost is $ \nu^{\star} = \gamma(z^{\star},z^{\star}) = k\mu + h(z^{\star})$. Proceed to steps~\ref{itm:step-2}--\ref{itm:step-5} if and only if $ K_{1} > 0 $.
\item \label{itm:step-2} Let
\begin{equation}
\Lambda(y) = \int_{-\infty}^{\infty}1_{(-\infty, y]}(g_{0}(u))\,\mathrm{d}u,
\label{eq:f0}
\end{equation}
where $ g_{0} $ is given by \eqref{eq:g0}. For $n=1,\ldots,N$, obtain $\hat{\nu}_{n}$ by solving the integral equation
\[
\int_{h(z^{\star})/\mu}^{-k+\hat{\nu}_{n}/\mu}\Lambda(u)\,\mathrm{d}u=K_{n}.
\]
Then, obtain $(\hat{s}_{n},\hat{S}_{n})$ by solving
\[
g_{0}(\hat{s}_{n})=g_{0}(\hat{S}_{n})=-k+\frac{\hat{\nu}_{n}}{\mu}.
\]
Put $\hat{\xi}_{n} = \hat{S}_{n}-\hat{s}_{n}$.
\item \label{itm:step-3} Define three index sets
\begin{align*}
\mathcal{N}_{<} & =\{n=1,\ldots,N:\hat{\xi}_{n}\leq Q_{n-1}\},\\
\mathcal{N}_{=} &  =\{n=1,\ldots,N:Q_{n-1}<\hat{\xi}_{n}<Q_{n}\},\\
\mathcal{N}_{>} & =\{n=1,\ldots,N:\hat{\xi}_{n}\geq Q_{n}\}.
\end{align*}
Put
\begin{equation}
\xi_{n}^{\star}=
\begin{cases}
Q_{n-1} & \mbox{if }n\in\mathcal{N}_{<},\\
\hat{\xi}_{n} & \mbox{if }n\in\mathcal{N}_{=},\\
Q_{n} & \mbox{if }n\in\mathcal{N}_{>}.
\end{cases}
\label{eq:Q-star}
\end{equation}
Then, define the candidate index set by
\begin{equation}
\mathcal{N}=\{n=1,\ldots,N:K(\xi_{n}^{\star})=K_{n}\}.
\label{eq:candidate-set}
\end{equation}
\item \label{itm:step-4} For $n\in\mathcal{N}_{=}$, let
\begin{equation}
s_{n}=\hat{s}_{n}, \quad S_{n}=\hat{S}_{n}, \quad \nu_{n}=\hat{\nu}_{n}.
\label{eq:hat-s-S-A}
\end{equation}
For $n\in\mathcal{N}\setminus\mathcal{N}_{=}$, obtain $(s_{n},S_{n})$
by solving the system of equations
\begin{equation}
\begin{cases}
S_{n}-s_{n}=\xi_{n}^{\star},\\
g_{0}(s_{n})=g_{0}(S_{n}),
\end{cases}
\label{eq:selected-system}
\end{equation}
and put
\begin{equation}
\nu_{n}=k\mu  + \frac{K(\xi_{n}^{\star})\mu}{\xi_{n}^{\star}} + \frac{\mu}
{\xi_{n}^{\star}}\int_{s_{n}}^{S_{n}}g_{0}(u)\,\mathrm{d}u.
\label{eq:gamma-n}
\end{equation}
\item \label{itm:step-5} Let
\begin{equation}
n^{\star}=\min\{n\in\mathcal{N}:\nu_{n}\leq \nu_{i}\mbox{ for all }i\in\mathcal{N}\}.
\label{eq:n-star}
\end{equation}
Taking
\begin{equation}
s^{\star}=s_{n^{\star}}\quad \mbox{and} \quad S^{\star}=S_{n^{\star}},
\label{eq:star}
\end{equation}
we obtain the optimal ordering policy, which is an $ (s,S) $ policy with reorder level $ s^{\star} $ and order-up-to level $ S^{\star} $. The minimum long-run average cost is $ \nu^{\star} = \gamma(s^{\star},S^{\star}) = \nu_{n^{\star}}$, where $ \nu_{n^{\star}} $ is given by \eqref{eq:gamma-n}--\eqref{eq:n-star}.
\end{enumerate}
The following corollary of Theorem~\ref{thm:main-results} states the optimality of the obtained ordering policy.

\begin{corollary}
\label{crly:algorithm}	
Assume that the setup cost $ K $ is given by \eqref{eq:setupcost} and that the holding and shortage cost $ h $ satisfies \textnormal{(\ref{itm:H1})--(\ref{itm:H5})}. If $ K_{1} = 0 $, the base stock policy $ U(z^{\star}, z^{\star}) $ is an optimal ordering policy that minimizes the long-run average cost, i.e., $ \nu^{\star} = \gamma(z^{\star}, z^{\star}) = k\mu + h(z^{\star}) $, where $ \gamma $ is given by \eqref{eq:gamma}. If $ K_{1} > 0 $, with $ (s^{\star}, S^{\star}) $ uniquely determined by steps~\ref{itm:step-1}--\ref{itm:step-5} of the above algorithm, $ U(s^{\star}, S^{\star}) $ is an optimal ordering policy that minimizes the long-run average cost, i.e., $ \nu^{\star} = \gamma(s^{\star}, S^{\star}) $.
\end{corollary}

Let us illustrate how the five-step algorithm yields the optimal policy. We obtain $ z^{\star} $, the minimizer of $ g_{0} $, at step~\ref{itm:step-1}. By Proposition~\ref{prop:z-s-S-star} and Theorem~\ref{thm:main-results}, the optimal policy is a base stock policy if $ K_{1} = 0 $, with $ z^{\star} $ being the optimal base stock level. If $ K_{1} > 0$, the optimal policy is of the $ (s,S) $ type with $ s<S $, and we obtain the optimal reorder and order-up-to levels by steps~\ref{itm:step-2}--\ref{itm:step-5}.

Assuming the setup cost is a constant $ K_{n} $ for any order quantity, we find the optimal reorder and order-up-to levels $ (\hat{s}_{n}, \hat{S}_{n}) $ for $ n = 1,\ldots,N $ in step~\ref{itm:step-2}. Under this policy, the quantity of each order is $ \hat{\xi}_{n} $ and the long-run average cost is $ \hat{\nu}_{n} $. The uniqueness and optimality of the obtained policy can be deduced from the following lemma.
\begin{lemma}
\label{lem:triple}
Let $ \kappa $ be a nonnegative number. Assume that $ K(\xi) = \kappa $ for all $ \xi > 0 $ and that $h$ satisfies \textnormal{(\ref{itm:H1})--(\ref{itm:H5})}. Then, there exists a unique $ \hat{\xi} \geq 0$ such that
\begin{equation}
\theta(\hat{\xi}) = \inf\{ \gamma(s,S) : s \leq S \},
\label{eq:theta-xi-hat}
\end{equation}
where $ \gamma $ is given by \eqref{eq:gamma} and $ \theta $ is given by \eqref{eq:theta}. In particular, $ \hat{\xi} = 0 $ if and only if $ \kappa = 0 $. Write
\begin{equation}
\hat{\nu} = \theta(\hat{\xi}), \quad \hat{s} = \tilde{s}(\hat{\xi}), \quad \hat{S} = \tilde{S}(\hat{\xi}),
\label{eq:hat}
\end{equation}
where $ \tilde{s} $ and $ \tilde{S} $ are defined by \eqref{eq:g0-equal}. Then, $ \hat{\nu} $ is the unique solution to
\begin{equation}
\int_{h(z^{\star})/\mu}^{-k+\hat{\nu}/\mu}\Lambda(u)\,\mathrm{d}u = \kappa
\quad \mbox{and} \quad \hat{\nu} \geq k\mu + h(z^{\star}),
\label{eq:nu}
\end{equation}
where $ \Lambda $ is given by \eqref{eq:f0} and $ (\hat{s}, \hat{S}) $ is the unique solution to 
\begin{equation}
g_{0}(\hat{s})=g_{0}(\hat{S})=-k+\frac{\hat{\nu}}{\mu}.
\label{eq:g0-solution}
\end{equation} 
Moreover, $ \hat{\xi} $, $\hat{S}$, and $ \hat{\nu} $ are strictly increasing in $\kappa$, whereas $\hat{s}$ is strictly decreasing in $\kappa$.
\end{lemma}

\begin{remark}
With a constant setup cost, \citet{Bather66} identified a set of necessary and sufficient conditions for the optimal $ (s,S) $ policy that minimizes the long-run average cost. Those conditions are equivalent to \eqref{eq:nu}--\eqref{eq:g0-solution} in Lemma~\ref{lem:triple}; see (4.2)--(4.4) and (5.4)--(5.5) in \citet{Bather66}. In particular, our Brownian control problem is reduced to Bather's problem when $ N = 1 $ in \eqref{eq:setupcost}.
\end{remark}

The next lemma is a technical result for proving Lemma~\ref{lem:triple} and Corollary~\ref{crly:algorithm}. It specifies how $ \theta $, the minimum average cost function, changes with the order quantity when the setup cost is assumed to be constant.

\begin{lemma}
\label{lem:theta-kappa}
Let $ \kappa $ be a positive number. Assume that $ K(\xi) = \kappa $ for all $ \xi > 0 $ and that $h$ satisfies \textnormal{(\ref{itm:H1})--(\ref{itm:H5})}. Then, there exists a unique $ \hat{\xi} > 0$ such that $ L(\hat{\xi}) = \kappa $ where
\[
L(\xi) = \int_{\tilde{s}(\xi)}^{\tilde{S}(\xi)}\big(  g_{0}(\tilde{s}(\xi))  -g_{0}(y)\big)  \,\mathrm{d}y.
\]
Moreover, the first derivative of $ \theta $ satisfies
\begin{equation}
\begin{cases}
\theta^{\prime}(\xi)<0 & \text{for }0< \xi < \hat{\xi},\\
\theta^{\prime}(\xi)=0 & \text{for }\xi = \hat{\xi},\\
\theta^{\prime}(\xi)>0 & \text{for }\xi > \hat{\xi}.
\end{cases}
\label{eq:theta-monotonicity}
\end{equation}
\end{lemma}

Let $K_{i}$ be the smallest one of $K_{1},\ldots,K_{N}$. By Lemma~\ref{lem:triple}, $\hat{\nu}_{i}$ is the smallest one of $\hat{\nu}_{1},\ldots,\hat{\nu}_{N}$. If $ K(\hat{\xi}_{i}) = K_{i} $ happens to hold, $U(\hat{s}_{i}, \hat{S}_{i})$ must be the optimal $(s,S)$ policy. The setup cost function in \eqref{eq:setupcost}, however, has imposed a constraint on order quantities for each setup cost value. When the setup cost is $K_{n}$, the quantity of an order is confined to an interval from $Q_{n-1}$ to $Q_{n}$ (which, by \eqref{eq:setupcost}, may be $(Q_{n-1},Q_{n})$, $(Q_{n-1},Q_{n}]$, $[Q_{n-1},Q_{n})$, or $[Q_{n-1},Q_{n}]$). If $\hat{\xi}_{i} < Q_{i-1}$ or $\hat{\xi}_{i} > Q_{i}$, we have $K(\hat{\xi}_{i}) \neq K_{i}$, and in either case, $U(\hat{s}_{i}, \hat{S}_{i})$ may not be optimal.

With a quantity-dependent setup cost, it is necessary to examine whether each $\hat{\xi}_{n}$ falls into the interval $(Q_{n-1},Q_{n})$; if not, we should adjust the order quantity to make it conform with \eqref{eq:setupcost}. At step~\ref{itm:step-3}, based on the relative position of $\hat{\xi}_{n}$ to the interval $(Q_{n-1},Q_{n})$, we define $\xi_{n}^{\star}$ as the point in $[Q_{n-1},Q_{n}]$ that is the closest to $\hat{\xi}_{n}$. By Lemma~\ref{lem:theta-kappa}, $\xi_{n}^{\star}$ is the optimal quantity when each order is confined in $[Q_{n-1},Q_{n}]$ with setup cost $K_{n}$. Consequently, one of $\xi_{1}^{\star},\ldots, \xi_{N}^{\star}$ must be the optimal order quantity for the setup cost given by \eqref{eq:setupcost}. We may thus seek the optimal policy by examining the policies that fix order quantities at $\xi_{n}^{\star}$ for $n=1,\ldots,N$. In this procedure, rather than examining all of $\xi_{1}^{\star},\ldots,\xi_{N}^{\star}$, we may just investigate those in the candidate index set $\mathcal{N}$ defined by \eqref{eq:candidate-set}. We will discuss the candidate index set shortly.

For $n\in\mathcal{N}$, we obtain the optimal $(s,S)$ policy with the quantity of each order fixed at $\xi_{n}^{\star}$. This task is carried out at step~\ref{itm:step-4}, where the reorder and order-up-to levels are given by $(s_{n},S_{n})$ and the resulting long-run average cost is given by $\nu_{n}$. When the quantity of each order is fixed at $\xi_{n}^{\star}$ with setup cost $K_{n}$, $U(\hat{s}_{n}, \hat{S}_{n})$ must be the optimal policy for $n\in\mathcal{N}_{=}$, and the long-run average cost is equal to $\hat{\nu}_{n}$. We may thus define $(s_{n}, S_{n}, \nu_{n})$ for $n \in \mathcal{N}_{=}$ by \eqref{eq:hat-s-S-A}. By Lemmas~\ref{lem:sQ}--\ref{lem:theta}, we can obtain $(s_{n}, S_{n}, \nu_{n})$ by solving \eqref{eq:selected-system}--\eqref{eq:gamma-n} for $n\in\mathcal{N} \setminus \mathcal{N}_{=}$. Note that not all of $\xi_{1}^{\star}, \ldots, \xi_{N}^{\star}$ are considered at step~\ref{itm:step-4}. The next lemma implies that it suffices to search for the optimal policy within the candidate index set $\mathcal{N}$. To state this lemma, let us define 
\begin{equation}
\underaccent{\bar}{\chi}(n) = \max\{j\in\mathcal{N} : j < n\}
\quad\mbox{for }n\in\mathcal{N}_{<}\setminus\mathcal{N}
\label{eq:chi-1}
\end{equation}
and
\begin{equation}
\bar{\chi}(n)  = \min\{j\in\mathcal{N} : j > n\}
\quad \mbox{for }
n\in\mathcal{N}_{>}\setminus\mathcal{N}.
\label{eq:chi-2}
\end{equation}
For $ n = 1, \ldots, N $, put 
\begin{equation}
\tilde{\nu}_{n} = \theta_{n}(\xi_{n}^{\star})
\label{eq:bar-gamma}
\end{equation}
where
\begin{equation}
\theta_{n}(\xi) = k\mu +  \frac{K_{n}\mu}{\xi} + \frac{\mu}{\xi} \int_{\tilde{s}(\xi)}^{\tilde{S}(\xi)} g_{0}(y)\,\mathrm{d}y
\quad \mbox{for $ \xi > 0 $.}
\label{eq:theta-n}
\end{equation} 
Note that $ \nu_{n} = \tilde{\nu}_{n} $ for $ n \in \mathcal{N} $.

\begin{lemma}
\label{lem:chi}
Assume that the setup cost	function in \eqref{eq:setupcost} has $K_{1}>0$ and that $h$ satisfies \textnormal{(\ref{itm:H1})--(\ref{itm:H5})}. Then, for each $n \in	\mathcal{N}_{<} \setminus \mathcal{N}$, $\underaccent{\bar}{\chi}(n)$ defined	by \eqref{eq:chi-1} exists and satisfies $ \nu_{\underaccent{\bar}{\chi}(n)} < \tilde{\nu}_{n}$; for each $n\in \mathcal{N}_{>} \setminus \mathcal{N}$, $\bar{\chi}_{n}$ defined by \eqref{eq:chi-2} exists and satisfies $\nu_{\bar{\chi}(n)} < \tilde{\nu}_{n} $.
\end{lemma}

It follows from Lemma~\ref{lem:chi} that $\nu^{\star} \leq \tilde{\nu}_{n}$ for all $n=1,\ldots, N$, where $\nu^{\star}$ is given by \eqref{eq:star} at step~\ref{itm:step-5}. Hence, $U(s^{\star},S^{\star})$ is the best candidate of the policies obtained by steps~\ref{itm:step-1}--\ref{itm:step-4}. Now let us prove the optimality of the obtained ordering policy.

\begin{proof}[Proof of Corollary~\ref{crly:algorithm}]
It suffices to show that $ (s^{\star},S^{\star}) $ obtained by steps~\ref{itm:step-1}--\ref{itm:step-5} satisfies \eqref{eq:s*S*}. Then, the corollary follows from Theorem~\ref{thm:main-results}.
	
Without loss of generality, we may assume $ K(\xi) = 0 $ for $ \xi \geq 0 $ if $ K_{1} = 0 $. Then, it follows from Lemma~\ref{lem:triple} that $ \hat{\xi} = 0 $ and thus $ \tilde{s}(0) = \tilde{S}(0) = z^{\star} $ is the optimal base stock level. Hence, $ s^{\star} = S^{\star} = z^{\star} $ and by \eqref{eq:gamma} and \eqref{eq:optimal-barrier}, $ \nu^{\star} = k\mu + h(z^{\star})$.
	
Consider the case $ K_{1} > 0$. If $ s = S $, it follows from \eqref{eq:gamma} that $ \gamma(s,S) = \infty $ because $ \ell = \infty $ by \eqref{eq:ell}. If $ s < S $, put $ \xi = S-s $ and assume that $ K(\xi) = K_{n} $ for $ n = 1,\ldots, N $. By \eqref{eq:Q-star}, \eqref{eq:theta-xi-hat}, and \eqref{eq:theta-monotonicity}, we obtain $ \theta_{n}(\xi) \geq \theta_{n}(\xi_{n}^{\star}) $, where $ \theta_{n} $ is given by \eqref{eq:theta-n}. If $ n \in \mathcal{N} $,
\[  
\gamma (s,S) \geq \theta_{n}(\xi) \geq \theta_{n}(\xi_{n}^{\star}) = \nu_{n} \geq \nu_{n^{\star}},
\]
where the first inequality follows from  \eqref{eq:gamma-tilde-definition} and \eqref{eq:theta-definition}, the first equality follows from \eqref{eq:g0-equal} and \eqref{eq:selected-system}--\eqref{eq:gamma-n}, and the last inequality follows from \eqref{eq:n-star}. If $ n \not\in \mathcal{N}$, we obtain
\[  
\gamma (s,S) \geq \theta_{n}(\xi) \geq \theta_{n}(\xi_{n}^{\star}) = \tilde{\nu}_{n} \geq \nu_{n^{\star}},
\]
where the equality follows from \eqref{eq:bar-gamma} and the last inequality follows from Lemma~\ref{lem:chi}.
\end{proof}

\section{Conclusion}
\label{sec:conclusion}

The optimality of $(s,S)$ policies for inventory systems with constant setup costs is a fundamental result in inventory theory. Assuming a Brownian demand process, we have extended the optimality of $(s,S)$ policies to stochastic inventory models with a general setup cost structure. To achieve this, we proved a comparison theorem that allows one to investigate the optimal policy within a tractable subset of admissible policies. When the setup cost is a step function, we proposed a policy selection procedure for obtaining the optimal control parameters. These results have improved the widely used lower bound approach for solving Brownian control problems and may apply to inventory models with even more general stochastic demand process, e.g., mean-reverting diffusions (see \citealt{CadenillasETAL10}) and spectrally positive L\'{e}vy processes (see \citealt{Kyprianou06} and \citealt{KuznetsovETAL12}). We look forward to exploring these extensions in future work.

\renewcommand{\theequation}{A.\arabic{equation}}
\setcounter{equation}{0}

\appendix

\section*{Technical proofs}

\begin{proof}[Proof of Lemma~\ref{lem:Z-dominance}]
By Lemma~\ref{lem:Skorokhod-Map}, $Z^{m}(t)\geq m$ for $t\geq 0$, so $Z(t)\leq Z^{m}(t)$ whenever $Z(t)\leq m$. For a fixed $t\geq 0$, if $Z(u) > m$ for all $u\in[0,t]$, we must have $Y(t)=0$ and thus $X(u)>m$ for all $u\in[0,t]$. It follows that $Y^{m}(t)=0$ and thus $Z(t)= Z^{m}(t)=X(t)$. If $Z(t)>m$ but there exits some $u\in[0,t)$ such that $Z(u)\leq m$, we put $t_{0}=\sup\{u\in[0,t): Z(u)\leq m\}$. We deduce that $Z(t_{0})\leq m$ because otherwise, $Z(t_{0})>m$ and $Z(t_{0}-)\leq m$, which contradicts the assumption that $Y\in\mathcal{U}_{m}$. Hence, $Z(t_{0})\leq Z^{m}(t_{0})$ and $t_{0}<t$. Because $Z(u)>m$ for $u\in(t_{0},t]$, $Y(t)-Y(t_{0})=0$. By \eqref{eq:Z}--\eqref{eq:X}, $Z(t)=Z(t_{0})+X(t)-X(t_{0})$. Because $Y^{m}$ has nondecreasing sample paths, 
\[
Z^{m}(t)=Z^{m}(t_{0})+X(t)-X(t_{0})+Y^{m}(t)-Y^{m}(t_{0})\geq Z(t).
\]
Therefore, $Z(t)\leq Z^{m}(t)$ for all $t\geq 0$.
\end{proof}
	
\begin{proof}[Proof of Lemma~\ref{lem:Exf}]
It suffices to consider $Y\in\mathcal{U}_{m}$ for a fixed positive integer $m$. Let $Z^{m}$ be the inventory process given by \eqref{eq:Z-m}, which is a reflected Brownian motion with lower reflecting barrier at $m$. Let $\alpha$ be a positive number. By \eqref{eq:psi},
\begin{align*}
\mathbb{E}_{x}[Z^{m}(t)^{\alpha}] &  =\alpha\int_{0}^{\infty}v^{\alpha-1}\psi_{x}^{m} (v,t)\,\mathrm{d}v
\leq \alpha\int_{0}^{x\vee m}v^{\alpha-1}\,\mathrm{d}v+\alpha\int_{x\vee m}^{\infty}v^{\alpha-1}\psi_{x}^{m}(v,t)\,\mathrm{d}v\\
&  \leq (x\vee m)^{\alpha}+\alpha\int_{x\vee m}^{\infty}v^{\alpha-1}\Phi \Big(  \frac{-v+(x\vee m)-\mu t}{\sigma t^{1/2}}\Big)  \,\mathrm{d}v+\alpha\int_{x\vee m}^{\infty}v^{\alpha-1} \mathrm{e}^{-\lambda(v-m)}\,\mathrm{d}v.
\end{align*}
For $t>0$ and $v>x\vee m$, we obtain
\[
\frac{-v+(x\vee m)-\mu t}{\sigma t^{1/2}} = -\Big(\frac{v-(x\vee m)}{\sigma t^{1/2}} + \frac{\mu t^{1/2}}{\sigma}\Big)\leq-\frac{2\mu^{1/2}\big(v-(x\vee m)\big)^{1/2}}{\sigma}
\]
using the inequality of arithmetic and geometric means. Therefore,
\[
\mathbb{E}_{x}[Z^{m}(t)^{\alpha}]\leq (x\vee m)^{\alpha} + \alpha\int_{x\vee m}^{\infty} v^{\alpha-1}\Phi\Big(  -\frac{2\mu^{1/2}\big(v-(x\vee m)\big)^{1/2}}{\sigma}\Big) \,\mathrm{d}v+\alpha\int_{x\vee m}^{\infty}v^{\alpha-1} \mathrm{e}^{-\lambda(v-m)}\,\mathrm{d}v
\]
for $t\geq 0$. All terms on the right side are finite and none of them depend on $t$, so
\begin{equation}
\sup_{t\geq 0} \mathbb{E}_{x}[Z^{m}(t)^{\alpha}] <
\infty\quad\mbox{for }\alpha>0.
\label{eq:ExZ-bound}
\end{equation}
	
Because $X(t)\leq Z(t)\leq Z^{m}(t)$ for $t\geq 0$,
\begin{equation}
|Z(t)|^{\alpha}\leq |X(t)|^{\alpha}+Z^{m}(t)^{\alpha}.
\label{eq:Z-bound}
\end{equation}
Since $X(t)$ follows a Gaussian distribution with mean $x-\mu t$ and variance $\sigma^{2}$,
\begin{equation}
\sup_{0\leq u\leq t}\mathbb{E}_{x}[|X(u)|^{\alpha}]<\infty\quad\mbox{for }t
\geq 0.
\label{eq:ExX}
\end{equation}
By \eqref{eq:fprime}, there exist $c_{0}>0$ and $c_{1}>0$ such that
\begin{equation}
|f(z)|<c_{0}+c_{1}|z|^{d+1}\quad\mbox{for }z\in\mathbb{R}.
\label{eq:f-bound}
\end{equation}
Then, we deduce that \eqref{eq:ExfZ1} holds from \eqref{eq:ExZ-bound}--\eqref{eq:f-bound} and that \eqref{eq:ExfZ2} holds from \eqref{eq:fprime}, \eqref{eq:ExZ-bound}--\eqref{eq:ExX}, and Tonelli's theorem. Since $Z(t)\leq Z^{m}(t)$ and $Z^{m}(t)\geq m$, it follows from \eqref{eq:f-bound} that
\[
\big|f(Z(t))\cdot 1_{[0,\infty)}(Z(t))\big| \leq c_{0}+c_{1} Z^{m}(t)^{d+1},
\]
which, along with \eqref{eq:ExZ-bound}, implies that \eqref{eq:limsup-fz} holds.
\end{proof}
	
\begin{proof}[Proof of Lemma~\ref{lem:g0}]
By (\ref{itm:H2})--(\ref{itm:H4}), $ \lim_{z\rightarrow\pm\infty}h(y+z)=\infty $, from which \eqref{eq:g0-limit} follows. By \eqref{eq:g0-prime},
\[
g_{0}^{\prime\prime}(z)=\frac{\lambda^{3}}{\mu}\int_{0}^{\infty} h(y+z) \mathrm{e}^{-\lambda y} \,\mathrm{d}y - \frac{\lambda^{2}}{\mu}h(z) - \frac{\lambda}{\mu} h^{\prime}(z) \quad\mbox{for $ z\neq0 $.}
\]
We would show that
\begin{equation}
g_{0}^{\prime}(z) > 0 \quad\mbox{for $ z\geq0 $,}
\label{eq:g-prime-1}
\end{equation}
\begin{equation}
g_{0}^{\prime\prime}(z)   >0\quad\mbox{for $ z<0 $,}
\label{eq:g-prime-2}
\end{equation}
and
\begin{equation}
\limsup_{z\rightarrow-\infty}g_{0}^{\prime}(z)<0.
\label{eq:g-prime-3}
\end{equation}
By these conditions and the continuity of $g_{0}^{\prime}$, we obtain \eqref{eq:z-star} with $z^{\star}<0$. Moreover, $z^{\star}$ is unique.
		
Write \eqref{eq:g0-prime} into
\begin{equation}
g_{0}^{\prime}(z)=\frac{\lambda^{2}}{\mu}\int_{0}^{\infty} (h(y+z)-h(z))\mathrm{e}^{-\lambda y}\,\mathrm{d}y.
\label{eq:g0-prime-2}
\end{equation}
Then, condition \eqref{eq:g-prime-1} follows from (\ref{itm:H4}). By (\ref{itm:H1}) and integration by parts,
\[
\int_{0}^{-z}h(y+z)\mathrm{e}^{-\lambda y}\,\mathrm{d}y=\frac{h(z)}{\lambda }+\frac{1}{\lambda}\int_{0}^{-z}h^{\prime}(y+z)\mathrm{e}^{-\lambda y}\,\mathrm{d}y\quad\mbox{for }z<0.
\]
It follows that
\[
g_{0}^{\prime\prime}(z)> \frac{\lambda^{3}}{\mu}\int_{0}^{-z}h(y+z) \mathrm{e}^{-\lambda y} \,\mathrm{d}y-\frac{\lambda^{2}}{\mu}h(z)-\frac{\lambda }{\mu}h^{\prime}(z) = \frac{\lambda^{2}}{\mu}\int_{0}^{-z}h^{\prime} (y+z)\mathrm{e}^{-\lambda y}\,\mathrm{d}y-\frac{\lambda}{\mu} h^{\prime}(z).
\]
Since $h$ is convex, $h^{\prime}(y)>h^{\prime}(z)$ for $y>z$. By (\ref{itm:H4}),
\[
g_{0}^{\prime\prime}(z)\geq\frac{\lambda^{2}}{\mu}\int_{0}^{-z}h^{\prime}(z) \mathrm{e}^{-\lambda y}\,\mathrm{d}y-\frac{\lambda }{\mu} h^{\prime}(z) = -\frac{\lambda }{\mu}h^{\prime}(z)\mathrm{e}^{\lambda z}>0,
\]
so \eqref{eq:g-prime-2} holds. By (\ref{itm:H2}) and (\ref{itm:H4}), there exist $z_{0}<0$ and $c_{0}>0$ such that $h^{\prime }(z)<-c_{0}$ for all $z<z_{0}$. Because $h$ is polynomially bounded,
\[
\lim_{z\rightarrow-\infty}\int_{z_{0}-z}^{\infty}(h(y+z)-h(z))\mathrm{e} ^{-\lambda y}\,\mathrm{d}y=0.
\]
Then by \eqref{eq:g0-prime-2},
\[
\limsup_{z\rightarrow-\infty}g_{0}^{\prime}(z)=\limsup_{z\rightarrow-\infty} \frac{\lambda^{2}}{\mu}\int_{0}^{z_{0}-z}(h(y+z) - h(z))\mathrm{e}^{-\lambda y}\,\mathrm{d}y\leq - \lim_{z\rightarrow-\infty}\frac{\lambda^{2}c_{0}}{\mu}\int_{0}^{z_{0}-z} \mathrm{e}^{-\lambda y}y\,\mathrm{d}y=-\frac{c_{0}}{\mu},
\]
which leads to \eqref{eq:g-prime-3}. 
\end{proof}

\begin{proof}[Proof of Lemma~\ref{lem:sQ}]
For $ \xi > 0 $ and $ s \in \mathbb{R} $, put
\[
G(s,\xi)=g_{0}(s+\xi)-g_{0}(s) = \int_{s}^{s+\xi}g'_{0}(y)\,\mathrm{d}y.
\]
By \eqref{eq:g0}, \eqref{eq:g0-prime}, and (\ref{itm:H3}), $ G $ is continuously differentiable on $ \mathbb{R} \times (0, \infty) $. If $G(s,\xi) = 0$, we must have $ s < z^{\star} < s+\xi $ by \eqref{eq:z-star}. Let $ \xi > 0 $ be fixed. Then, $G(s,\xi)$ is continuous and strictly increasing in $s$ on $[z^{\star}-\xi,z^{\star}]$, with $ G(z^{\star}-\xi, \xi) < 0 $ and $G(z^{\star}, \xi) > 0$. Hence, there exists a unique $ s = \tilde{s}(\xi) $ such that $ G(s, \xi) = 0 $, by which we deduce that both \eqref{eq:g0-equal} and \eqref{eq:sQ} hold. The limits in  \eqref{eq:s-tilde-infinity} follows from \eqref{eq:g0-limit} and \eqref{eq:g0-equal}. Using \eqref{eq:z-star} again, we obtain
\[  
\frac{\partial }{\partial s} G ( \tilde{s}(\xi), \xi ) =  g^{\prime}_{0}(\tilde{S}(\xi))-g^{\prime}_{0}(\tilde{s}(\xi)) > 0 \quad \mbox{for $ \xi > 0 $.}
\]
By the implicit function theorem (see, e.g., Theorem~11.1 in \citealt{Protter98}), $ \tilde{s} $ must be a differentiable function on $ (0, \infty) $. By \eqref{eq:z-star}--\eqref{eq:sQ},
\[ 
\frac{\partial}{\partial \xi} G (\tilde{s}(\xi), \xi) = g_{0}^{\prime} (\tilde{s}(\xi) + \xi) = g_{0}^{\prime}(\tilde{S}(\xi)) > 0.
\]
The implicit function theorem also implies that
\[  
\tilde{s}^{\prime}(\xi) = - \frac{\partial G (\tilde{s}(\xi), \xi) / \partial \xi }{\partial G (\tilde{s}(\xi), \xi) / \partial s} = - \frac{g^{\prime}_{0}(\tilde{S}(\xi))}{g^{\prime}_{0}(\tilde{S}(\xi))-g^{\prime}_{0}(\tilde{s}(\xi))} < 0.
\]
It follows from \eqref{eq:z-star}--\eqref{eq:sQ} that
\[  
\tilde{S}^{\prime}(\xi) = \tilde{s}^{\prime}(\xi) + 1 = - \frac{g^{\prime}_{0}(\tilde{s}(\xi))}{g^{\prime}_{0}(\tilde{S}(\xi))-g^{\prime}_{0}(\tilde{s}(\xi))} > 0.
\]
\end{proof}
	
\begin{proof}[Proof of Lemma~\ref{lem:theta}]
By \eqref{eq:z-star} and \eqref{eq:gamma-tilde}, $ \theta(0) = \tilde{\gamma}(z^{\star}, 0) = \tilde{\gamma} ( \tilde{s}(0), 0 ) $, so \eqref{eq:theta} holds for $ \xi = 0 $. For $ \xi > 0 $, Lemma~\ref{lem:sQ} implies that $ s = \tilde{s}(\xi) $ is the unique solution to $ \partial \tilde{\gamma}(s,\xi) / \partial s = 0$. By \eqref{eq:z-star}--\eqref{eq:sQ}, \[  
\frac{\partial^{2} }{\partial s^{2}} \tilde{\gamma}( \tilde{s}(\xi), \xi ) = \frac{\mu}{\xi} g^{\prime}_{0}(\tilde{S}(\xi)) - \frac{\mu}{\xi} g^{\prime}_{0}(\tilde{s}(\xi)) > 0,
\]
so \eqref{eq:theta} also holds for $ \xi > 0 $. 
		
By \eqref{eq:ell} and \eqref{eq:gamma-tilde},
\[  
\liminf_{\xi \downarrow 0} \theta (\xi) = (k + \ell)\mu + \liminf_{\xi \downarrow 0} \frac{\mu}{\xi} \int_{\tilde{s}(\xi)}^{\tilde{S}(\xi)} g_{0}(y)\,\mathrm{d}y.
\]
Since \eqref{eq:sQ} implies that $ \tilde{s}(0+) = z^{\star} $, we obtain $ \liminf_{\xi \downarrow 0} \theta (\xi) = \theta(0) $, so $ \theta $ is lower semicontinuous at zero. For $ \tilde{\xi} > 0 $, since $ K $ is lower semicontinuous, by Proposition~B.1 in \citet{Puterman94},
\[  
\frac{K(\tilde{\xi})}{\tilde{\xi}} \leq \liminf_{\xi \rightarrow \tilde{\xi}} \frac{K(\xi)}{\xi}.
\]
By Lemma~\ref{lem:sQ}, $ \tilde{s} $ is continuous on $ [0,\infty) $, by which we obtain
\[  
\int_{\tilde{s}(\tilde{\xi})}^{\tilde{S}(\tilde{\xi})} g_{0}(y)\,\mathrm{d}y = \lim_{\xi \rightarrow \tilde{\xi}} \int_{\tilde{s}(\xi)}^{\tilde{S}(\xi)} g_{0}(y)\,\mathrm{d}y.
\]
It follows that $ \theta(\tilde{\xi}) \leq \liminf_{\xi \rightarrow \tilde{\xi}} \theta (\xi)$, and thus $ \theta $ is lower semicontinuous on $ [0,\infty) $.
		
By \eqref{eq:s-tilde-infinity} and L'H\^{o}pital's rule,
\[  
\lim_{\xi \rightarrow \infty} \frac{1}{\xi}  \int_{\tilde{s}(\xi)}^{\tilde{S}(\xi)} g_{0}(y)\,\mathrm{d}y = \lim_{\xi \rightarrow \infty} ( \tilde{s}^{\prime}(\xi)+1 ) g_{0}( \tilde{S}(\xi) ) - \lim_{\xi \rightarrow \infty} \tilde{s}^{\prime}(\xi) g_{0}( \tilde{s}(\xi) ).
\]
Then using \eqref{eq:g0-limit}, \eqref{eq:g0-equal}, and \eqref{eq:s-tilde-infinity}, we obtain
\[  
\lim_{\xi \rightarrow \infty} \frac{1}{\xi}  \int_{\tilde{s}(\xi)}^{\tilde{S}(\xi)} g_{0}(y)\,\mathrm{d}y = \lim_{\xi \rightarrow \infty} g_{0}( \tilde{S}(\xi) ) = \infty.
\]
Because the setup cost is nonnegative, the above limit implies that $ \lim_{\xi \rightarrow \infty} \theta(\xi) = \infty $. 
\end{proof}
	
\begin{proof}[Proof of Lemma~\ref{lem:bar-gamma}]
Put $ \bar{K} = \sup\{ K(\xi) : \xi > 0 \}$, which is finite by (\ref{itm:S2}). By \eqref{eq:theta-infinity}, there exists $ 0 < \bar{\xi} < \infty $ such that 
\begin{equation}
\theta(\bar{\xi}) > \frac{\bar{K}\mu}{\bar{\xi}} + \gamma(s^{\star},S^{\star}). 
\label{eq:theta-bar}
\end{equation}
Take $ \underaccent{\bar}{s} = \tilde{s}(\bar{\xi}) $. If $ s \geq \underaccent{\bar}{s} $, $ \bar{\gamma}(s,\xi) = \tilde{\gamma}(s,\xi) $ by \eqref{eq:gamma-tilde}. Then, \eqref{eq:gamma-tilde-bound} follows from \eqref{eq:s*S*} and \eqref{eq:gamma-tilde-definition}. 
		
It remains to prove \eqref{eq:gamma-tilde-bound} for $ s < \underaccent{\bar}{s} $, which relies on the following inequalities below deduced from \eqref{eq:z-star}--\eqref{eq:sQ}, 
\[
\begin{cases}
g_{0}(y) < g_{0}(\underaccent{\bar}{s}) = g_{0}(\underaccent{\bar}{s} + \bar{\xi} ) & \mbox{for $ \underaccent{\bar}{s} < y < \underaccent{\bar}{s}+\bar{\xi}$,} \\
g_{0}(y) > g_{0}(\underaccent{\bar}{s}) = g_{0}(\underaccent{\bar}{s} + \bar{\xi} ) & \mbox{for $ y < \underaccent{\bar}{s} $ or $ y > \underaccent{\bar}{s}+\bar{\xi}$.}
\end{cases}
\]
If $ s \leq \underaccent{\bar}{s} - \xi $, 
\begin{equation}
\frac{1}{\xi} \int_{s}^{s+\xi} g_{0}(y \vee \underaccent{\bar}{s})\,\mathrm{d}y =  g_{0}(\underaccent{\bar}{s}) > \frac{1}{\bar{\xi}} \int_{\underaccent{\bar}{s}}^{\underaccent{\bar}{s}+\bar{\xi}} g_{0}(y)\,\mathrm{d}y.
\label{eq:g0-bar}
\end{equation}
Then by \eqref{eq:theta} and \eqref{eq:theta-bar},
\[ 
\bar{\gamma}(s,\xi) > k\mu +\frac{K(\xi)\mu}{\xi} + \frac{\mu}{\bar{\xi}} \int_{\underaccent{\bar}{s}}^{\underaccent{\bar}{s}+\bar{\xi}} g_{0}(y)\,\mathrm{d}y = \theta(\bar{\xi}) +\frac{K(\xi)\mu}{\xi} - \frac{K(\bar{\xi})\mu}{\bar{\xi}} > \gamma(s^{\star}, S^{\star}). 
\]
If $ \underaccent{\bar}{s} - \xi < s < \underaccent{\bar}{s} \wedge (\underaccent{\bar}{s} + \bar{\xi} - \xi) $, 
\[  
\int_{s}^{s+\xi} g_0(y \vee \underaccent{\bar}{s})\,\mathrm{d}y = \int_{\underaccent{\bar}{s}}^{s+\xi} g_0(y)\,\mathrm{d}y + (\underaccent{\bar}{s} - s) g_{0}(\underaccent{\bar}{s}) > \int_{\underaccent{\bar}{s}}^{\underaccent{\bar}{s}+\xi} g_0(y)\,\mathrm{d}y,
\]
which implies that $ \bar{\gamma}(s,\xi) > \tilde{\gamma}(\underaccent{\bar}{s},\xi) $, and thus \eqref{eq:gamma-tilde-bound} follows. If $ \underaccent{\bar}{s} + \bar{\xi} - \xi \leq s < \underaccent{\bar}{s} $,
\begin{align*} 
\int_{s}^{s+\xi} g_0(y \vee \underaccent{\bar}{s})\,\mathrm{d}y & > (\underaccent{\bar}{s} - s) g_{0}(\underaccent{\bar}{s}) + \int_{\underaccent{\bar}{s}}^{\underaccent{\bar}{s}+\bar{\xi}} g_0(y)\,\mathrm{d}y + (s + \xi - \underaccent{\bar}{s} - \bar{\xi} ) g_{0}(\underaccent{\bar}{s} + \bar{\xi}) \\
& = \int_{\underaccent{\bar}{s}}^{\underaccent{\bar}{s}+\bar{\xi}} g_0(y)\,\mathrm{d}y + (\xi - \bar{\xi}) g_{0}(\underaccent{\bar}{s}),
\end{align*}
where the last equality follows from \eqref{eq:g0-equal}. Then,
\[  
\frac{1}{\xi} \int_{s}^{s+\xi} g_0(y \vee \underaccent{\bar}{s})\,\mathrm{d}y = \frac{1}{\xi}\int_{\underaccent{\bar}{s}}^{\underaccent{\bar}{s}+\bar{\xi}} g_0(y)\,\mathrm{d}y + \frac{\xi - \bar{\xi}}{\xi}  g_{0}(\underaccent{\bar}{s}) > \frac{1}{\bar{\xi}} \int_{\underaccent{\bar}{s}}^{\underaccent{\bar}{s}+\bar{\xi}} g_{0}(y)\,\mathrm{d}y,
\]
where the last inequality follows from \eqref{eq:g0-bar}. Since the above inequality is identical to \eqref{eq:g0-bar}, we deduce that \eqref{eq:gamma-tilde-bound} holds. 
\end{proof}
	
\begin{proof}[Proof of Lemma~\ref{lem:Zm}]
Clearly, $Y_{m}\in\mathcal{U}_{m}$. Both from $Z_{m}(0-)=Z(0-)=x$, these two inventory levels satisfy $Z_{m}(0)\leq Z(0)$ by (\ref{itm:J1})--(\ref{itm:J4}). For $t\geq0$, let 
\[
t_{0}=\sup\{u\in[0,t]:u\in I_{m}\setminus J,\ \Delta Y_{m}(u)>0\},
\]
with the convention $\sup\varnothing=0$. Then, $Z_{m}(t_{0})\leq Z(t_{0})$ by (\ref{itm:J4}). If there is some $u\in(t_{0},t]$ for which $\Delta Y_{m}(u)>0$, it is of type (\ref{itm:J2}) or (\ref{itm:J3}) and thus $\Delta Y_{m}(u)\leq\Delta Y(u)$. Because $Y^{c}$ is nondecreasing, it follows from \eqref{eq:Ymc} that $Y_{m}^{c}(t) - Y_{m}^{c}(t_{0}) \leq Y^{c}(t)-Y^{c}(t_{0})$. Hence, $Z_{m}(t)\leq Z(t)$ for $t\geq0$.
		
Consider some $t\geq0$ for which $Z_{m}(t)<0$. If $x\leq m/2$ and $Z_{m}(u) \leq m/2$ for all $u\in[0,t]$, we have $Y_{m}^{c}(t)=Y^{c}(t)$ by \eqref{eq:Ymc} and $\Delta Y_{m}(u)=\Delta Y(u)$ for all $u\in[0,t]$ by (\ref{itm:J2}). Hence, $Z_{m}(t)=Z(t)$. If there exists some $t_{1}\in[0,t)$ such that $Z_{m}(t_{1})> m/2$, let us consider the time
\[
t_{2}=\sup\{u\in[ t_{1},t):Z_{m}(u)>0\}.
\]
Since $Z_{m}$ does not have downward jumps, we have $Z_{m}(t_{2}-)=0$ and $\Delta Y_{m}(t_{2})=0$, which yields $Z_{m}(t_{2})=0$. Then, $Z(t_{2})\geq0$ because $Z_{m}(t_{2})\leq Z(t_{2})$. If $Z(t_{2})>0$, we can deduce from (\ref{itm:J2})--(\ref{itm:J4}) that $\Delta Y_{m}(t_{2})>0$, a contradiction. Hence, $Z_{m}(t_{2}) = Z(t_{2})=0$. Because $Z_{m}(u)\leq0$ for all $u\in[t_{2},t]$, we have $Y_{m}^{c}(t)-Y_{m}^{c}(t_{2})=Y^{c}(t)-Y^{c}(t_{2})$ by \eqref{eq:Ymc} and $\Delta Y_{m}(u)=\Delta Y(u)$ for all $u \in [t_{2},t]$ by (\ref{itm:J2}). It follows that $Z_{m}(t)=Z(t)<0$ holds. 
\end{proof}
	
\begin{proof}[Proof of Lemma~\ref{lem:triple}]
Let us first prove the uniqueness and monotonicity of the solutions to \eqref{eq:nu}--\eqref{eq:g0-solution}. By \eqref{eq:z-star}, $ g_{0} $ has the minimum value at $ z^{\star} $. Put
\[
I(u)=\int_{g_{0}(z^{\star})}^{u}\Lambda(y)\,\mathrm{d}y \quad \mbox{for $ u \geq g_{0}(z^{\star}) $.}
\]
Then, $I(u)$ is a continuous function of $u$, with $I( g_{0}(z^{\star}))=0$. Because $\Lambda(y)$ is nondecreasing in $y$ and $\Lambda (y) > 0$ for $y > g_{0}(z^{\star}) $, $I(u)$ is strictly increasing in $u$ when $u > g_{0}(z^{\star})$ and $I(u)\rightarrow\infty$ as $u\rightarrow \infty$. Hence, for each $ \kappa \geq 0 $, there is a unique $ \hat{u} \geq g_{0}(z^{\star}) $ such that $I(\hat{u}) = \kappa$. In addition, $\hat{u}$ is strictly increasing in $\kappa$. By \eqref{eq:optimal-barrier} and \eqref{eq:g0}, $ g_{0}(z^{\star}) = h(z^{\star})/\mu $, so the solution to \eqref{eq:nu} is unique and $ \hat{\nu} $ is strictly increasing in $ \kappa $. Note that $ g_{0}(z^{\star}) \leq -k + \hat{\nu}/\mu $. The uniqueness of the solution to \eqref{eq:g0-solution} also follows from \eqref{eq:z-star}. Moreover, $ \hat{s} $ is strictly decreasing in $ \hat{\nu} $ and $ \hat{S} $ is strictly increasing in $ \hat{\nu} $. Then, their monotonicity in $ \kappa $ follows from that of $ \hat{\nu} $. The monotonicity of $ \hat{\xi} $ in $ \kappa $ follows from the fact that $ \hat{\xi} = \hat{S} - \hat{s}$.
		
Next, let us prove the optimality of $ (\hat{s}, \hat{S}) $. When $ \kappa = 0 $, by \eqref{eq:gamma-tilde} and \eqref{eq:theta},
\[  
\theta(\xi) = 
\begin{dcases}
k\mu + \frac{\mu}{\xi}\int_{\tilde{s}(\xi)}^ {\tilde{S}(\xi)} g_{0}(y)\,\mathrm{d}y & \mbox{for $ \xi > 0 $,}\\
k\mu + \mu g_{0}(z^{\star}) & \mbox{for $ \xi = 0 $.}
\end{dcases}
\]
By \eqref{eq:z-star}, $ \theta (0) < \theta (\xi)$ for $ \xi > 0 $, so $ \hat{\xi} = 0 $ is the unique solution to \eqref{eq:theta-xi-hat}. If $ \kappa > 0 $, we obtain $ \theta(0) = \infty $ by \eqref{eq:gamma-tilde} since $ \ell = \infty $. Hence, $ \hat{\xi} = 0 $ if and only if $ \kappa = 0 $. By \eqref{eq:optimal-barrier} and \eqref{eq:g0}, $ \hat{\nu} = k\mu + h(z^{\star})$ and $ \hat{s} = \hat{S} = z^{\star} $, and they satisfy \eqref{eq:nu} and \eqref{eq:g0-solution}, respectively.
		
When $ \kappa > 0 $, the uniqueness of $ \hat{\xi} $ follows from Lemma~\ref{lem:theta-kappa}. It remains to show $ \hat{\nu} $ satisfies \eqref{eq:nu} and $ (\hat{s}, \hat{S}) $ satisfies \eqref{eq:g0-solution}. By \eqref{eq:g0-equal}, \eqref{eq:gamma-tilde}, \eqref{eq:theta}, and the fact that $ L(\hat{\xi}) = \kappa $, we obtain 
\[
\theta(\hat{\xi})=k\mu+\mu g_{0}(  \tilde{s}(\hat{\xi})).
\]
Then, \eqref{eq:g0-solution} follows from \eqref{eq:g0-equal} and \eqref{eq:hat}. By Tonelli's theorem and \eqref{eq:z-star},
\[  
L(\xi) = \int_{g_{0}(z^{\star})}^{g_{0}( \tilde{s}(\xi))} \int_{\tilde{s}(\xi)}^{\tilde{S}(\xi)} 1_{(-\infty, y]}(g_{0}(u)) \,\mathrm{d}u \,\mathrm{d}y = \int_{g_{0}(z^{\star})}^{g_{0}( \tilde{s}(\xi))} \int_{-\infty}^{\infty} 1_{(-\infty, y]}(g_{0}(u)) \,\mathrm{d}u \,\mathrm{d}y = I(g_{0}( \tilde{s}(\xi))).
\]
Using \eqref{eq:g0-solution} and the fact that $ L(\hat{\xi}) = \kappa $, we obtain $ I(-k + \hat{\nu}/\mu ) = \kappa $, so $ \hat{\nu} $ satisfies \eqref{eq:nu}.
\end{proof}
	
\begin{proof}[Proof of Lemma~\ref{lem:theta-kappa}]
By \eqref{eq:z-star}--\eqref{eq:g0-equal}, $g_{0}(y) < g_{0}( \tilde{s}(\xi) ) = g_{0} ( \tilde{S}(\xi) ) $ for $ \tilde{s}(\xi) < y < \tilde{S}(\xi) $. Then by \eqref{eq:g0-equal} and \eqref{eq:s-tilde-prime}, $ L $ is continuous and strictly increasing, with $ L(0) = 0$. By \eqref{eq:g0-limit} and \eqref{eq:s-tilde-infinity}, $ L(\xi) \rightarrow \infty $ as $ \xi \rightarrow \infty $. It follows that for each $ \kappa > 0 $, there is a unique $ \hat{\xi} > 0 $ such that $ L(\hat{\xi}) = \kappa $.
		
By \eqref{eq:g0-equal}, \eqref{eq:gamma-tilde}, and \eqref{eq:theta}, 
\[
\theta(\xi) = k\mu + \frac{\kappa\mu}{\xi}+\frac{\mu}{\xi}\int_{\tilde{s}(\xi)}^ {\tilde{S}(\xi)} g_{0}(y)\,\mathrm{d}y,
\]
the first derivative of which is
\[
\theta^{\prime}(\xi) = -\frac{\kappa\mu}{\xi^{2}} - \frac{\mu}{\xi^{2}} \int_{\tilde{s}(\xi)}^{\tilde{S}(\xi)} g_{0}(y) \,\mathrm{d}y + \frac{\mu} {\xi}\big(  g_{0}(\tilde{S}(\xi))(  \tilde{s}^{\prime} (\xi)+1) - g_{0}(  \tilde{s}(\xi)) \tilde{s}^{\prime} (\xi)\big).
\]
Using \eqref{eq:g0-equal} again, we obtain
\[
\theta^{\prime}(\xi)=-\frac{\kappa\mu}{\xi^{2}}-\frac{\mu}{\xi^{2}} \int_{\tilde{s}(\xi)}^{\tilde{S}(\xi)}g_{0}(y)\,\mathrm{d}y+\frac{\mu}{\xi} g_{0}(  \tilde{s}(\xi))  =\frac{\mu(  L(\xi)-\kappa)  }{\xi^{2}}.
\]
Then, \eqref{eq:theta-monotonicity} follows from the fact that $ L(\hat{\xi}) = \kappa $ and the monotonicity of $ L $.
\end{proof}
	
\begin{proof}[Proof of Lemma~\ref{lem:chi}]
Suppose that there exists some $n\in\mathcal{N}_{<}\setminus\mathcal{N}$ such 	that $\underaccent{\bar}{\chi}(n)$ does not exist, i.e., $K(\xi_{j}^{\star})\neq	K_{j}$ for $j=1,\ldots,n-1$. Since $K_{1}>0$, Lemma~\ref{lem:triple} implies that 	$\hat{\xi}_{1}>0=Q_{0}$, so $1\notin\mathcal{N}_{<}$ and $n\geq2$. Because 	$\xi_{n}^{\star}=Q_{n-1}$ and $K(\xi_{n}^{\star})\neq K_{n}$, $K(\xi_{n}^{\star}) = 	K(Q_{n-1})=K_{n-1}$. If $n-1\in\mathcal{N}_{>}$, we should have	$K(\xi_{n-1}^{\star}) = K(Q_{n-1})=K_{n-1}$, contradicting the hypothesis that $K(\xi_{n-1}^{\star})\neq K_{n-1}$. It follows that $n-1\in\mathcal{N}_{<}\setminus	\mathcal{N}$. By induction, we obtain $\{1,\ldots,n-1\} \subset	\mathcal{N}_{<}\setminus \mathcal{N}$, which contradicts the fact that	$1\notin\mathcal{N}_{<}$. Hence, $\underaccent{\bar}{\chi}(n)$ must exist. 
		
For $n\in\mathcal{N}_{<}\setminus\mathcal{N}$, the above arguments also imply that $\{\underaccent{\bar}{\chi}(n)+1, \ldots, n\} \subset \mathcal{N}_{<}\setminus \mathcal{N}$, which yields $ K_{\underaccent{\bar}{\chi} (n)} < \cdots<K_{n} $. By Lemma~\ref{lem:triple}, $
\hat{\xi}_{\underaccent{\bar} {\chi}(n)} < \hat{\xi}_{\underaccent{\bar}{\chi}(n)+1} \leq Q_{\underaccent{\bar} {\chi}(n)}$, so $\underaccent{\bar}{\chi}(n) \in \mathcal{N}_{=}\cup\mathcal{N}_{<}$. It follows that
\[  
\tilde{\nu}_{n} = \theta_{n}(Q_{n-1}) > \theta_{\underaccent{\bar}{\chi}(n)} (Q_{n-1}) > \theta_{\underaccent{\bar}{\chi}(n)} (\xi_{\underaccent{\bar}{\chi}(n)}^{\star}) = \tilde{\nu}_{\underaccent{\bar}{\chi}(n)},
\]
where the first equality is due to the fact that $ \xi_{n}^{\star} = Q_{n-1}$, the first inequality is due to the fact that $ K_{\underaccent{\bar}{\chi} (n)} <  K_{n} $, and the second inequality is due to \eqref{eq:theta-monotonicity} and the fact that 
\[ 
\hat{\xi}_{\underaccent{\bar}{\chi}(n)} \leq \xi_{\underaccent{\bar}{\chi}(n)}^{\star} < Q_{\underaccent{\bar}{\chi}(n)} \leq Q_{n-1} . 
\]
Since $ \tilde{\nu}_{j} = \nu_{j} $ for $ j \in \mathcal{N} $, we obtain $ \tilde{\nu}_{n} > \nu_{\underaccent{\bar}{\chi}(n)} $.
		
Using the fact that $\hat{\xi}_{N}<Q_{N}=\infty$, we can follow similar arguments 	to prove that $\bar{\chi}(n)$ exists and that $\nu_{\bar{\chi}(n)} < \tilde{\nu}_{n}$ for $n \in \mathcal{N}_{>} \setminus \mathcal{N} $. The details are thus omitted.
\end{proof}

\section*{Acknowledgments}

The work of S.\ He was supported in part by MOE AcRF under grant R266000086112 and by NUS Global Asia Institute under grant R716000006133. The work of D.\ Yao was supported in part by the National Natural Science Foundation of China under grant 11401566.

\bibliographystyle{ormsv080}
\bibliography{refs}

\end{document}